\documentclass{ws-m3as}

\usepackage{comment}

\usepackage{algorithm}
\usepackage{algorithmicx}
\usepackage{algpseudocode}

\usepackage{subcaption}
\usepackage{float}   
\usepackage{caption} 

\usepackage{xcolor}

\date{}

\begin{document}

\markboth{Jingyuan Hu, Zhongjian Wang, Jack Xin, and Zhiwen Zhang}{Particle-Field Algorithm for RDA Cancer Cell Invasion}

\title{A Novel Stochastic Particle-Field Algorithm for a Reaction-Diffusion-Advection Cancer Invasion Model}

\author{Jingyuan Hu}

\address{Department of Mathematics, The University of Hong Kong, Pokfulam Road\\
Hong Kong SAR,
P.R.China\\
hujy@connect.hku.hk}

\author{Zhongjian Wang}

\address{Division of Mathematical Sciences, School of Physical and Mathematical Sciences, \\Nanyang Technological University, 21 Nanyang Link\\
Singapore 637371, Singapore\\ zhongjian.wang@ntu.edu.sg}

\author{Jack Xin}

\address{Department of Mathematics, University of California at Irvine, Irvine
\\CA 92697, USA\\
jack.xin@uci.edu}

\author{Zhiwen Zhang
\footnote{Corresponding author.}}
\address{Department of Mathematics, The University of Hong Kong, Pokfulam
Road
\\Hong Kong SAR, P.R.China
\\Materials Innovation Institute for Life Sciences and Energy
(MILES), HKU-SIRI
\\Shenzhen, 518045, P.R. China.
\\zhangzw@hku.hk}

\maketitle

\begin{abstract}
In this paper, we present a novel numerical framework for solving a specific biological reaction-diffusion-advection (RDA) system of cancer growth in three dimensions (3D) using particles of variable mass. We adopt empirical particle measures to represent cell density and dynamically construct the concentration fields of multiple related chemical species throughout the 3D domain. Efficient interaction between the particles and the spatial grid is achieved through a Particle-in-Cell (PIC) algorithm, while diffusion in space is solved rapidly using a spectral method. We demonstrate that for this particular system, the rate of change of particle mass remains bounded over finite time intervals. Furthermore, in addition to the inherent positivity preservation of cell density guaranteed by the empirical particle measures, the concentrations constructed by the algorithm are also unconditionally positivity-preserving on the spatial grid. Moreover, we present a rigorous error analysis for the proposed method, and numerical experiments confirm the theoretical convergence rates. To the best of our knowledge, this is the first numerical work to solve this system in three dimensions, wherein a rapid spread of cells driven by haptotactic flux is observed, similar to the behavior documented in the two-dimensional case.
\end{abstract}



\section{Introduction}
Reaction-diffusion-advection (RDA) systems have played a central role in modeling and simulating cancer cell invasion into healthy tissues over the past decades. Besides being flexible to account for various biophysical mechanisms, RDA systems can generate synthetic data for theoretical understanding and therapeutic intervention in the absence of wet-lab experimental data. While mesh-based methods (e.g, finite difference\cite{chertock2008second,Wise_2008}) have been a classical choice for computation, they can be expensive when solutions contain large gradients to be resolved, especially in three space dimensions, and may also complicate coupling with agent-based discrete dynamics for individual cell behavior\cite{Cont-Agent-Model-1998,hybrid_2017}. 

In this paper, we develop a novel stochastic interacting particle-field (SIPF) method for RDA cancer growth models. Specifically, we work with the following system studied in\cite {chertock2008second,anderson2005hybrid,cancer_2006}. The system couples four dependent variables: $u(\mathbf{x},t)$ denotes the cell density, $v(\mathbf{x},t)$ represents the density of the extracellular matrix (ECM), $m(\mathbf{x},t)$ is the concentration of the matrix-degrading enzyme (MDE), and $w(\mathbf{x},t)$ stands for the oxygen concentration, in a bounded spatial domain $\mathbf{x} \in \Omega \subset \mathbb{R}^d$. The spatio-temporal dynamics are governed by:
\begin{equation}\label{eq:haptotaxis-system}
    \begin{cases}
        u_t + \chi \nabla \cdot (u \nabla v) = D_u \Delta u - u + \rho(w) u, \\
        v_t = - \alpha m v, \\
        m_t = D_m \Delta m - \beta m + u, \\
        w_t = D_w \Delta w - \eta(u) w + \gamma v - w,\quad \mathbf{x} \in \Omega,\ t \in [0,T]
    \end{cases}
\end{equation}
where $\chi, D_u, D_m, D_w, \alpha, \beta, \gamma$ are positive constants. In the equation of \(u\), the advection term $\chi \nabla \cdot (u \nabla v)$ models cell movement directed along gradients of the concentration of the ECM $\nabla v$, the diffusion term \(D_u \Delta u\) models random movements, the decay term \(-u\) models cell death, and the proliferation function $\rho(w):=\frac{2w}{1+w}$ models cell proliferation in response to oxygen availability, where a nonlinear saturation effect occurs at large $w$. In the equation of \(v\), the decay term \(-\alpha m v\) models the degradation effect proportional to the concentration of MDE \(m\). In the equation of \(m\), \(D_m \Delta m -\beta m\) models the diffusion and decay of MDE concentration, and \(u\) models the production from cells. In the equation of \(w\), \(D_w \Delta w - w\) models the diffusion and decay of oxygen concentration, \(\gamma v\) models the production from ECM, and the oxygen consumption function $\eta(u):=\frac{2u}{1+u}$ describes another nonlinear saturation effects at high cell density \(u\). 

In the $u$ equation, the growth $\rho(w)\, u$ and decay ($-u$) terms, arising from cell (death) birth with (in)sufficient oxygen, lead to a lack of conservation of total mass of $u$ and new challenges to the particle approximation of $u$. In the absence of growth/decay
(low order) terms in the $u$ equation\cite{anderson2005hybrid} and without the $w$ component, the simplified system resembles the parabolic-parabolic Keller-Segel chemotaxis system\cite{SIPF1,hu2025fast} and an SIPF method has been developed in\cite{Hu2024}.
In this case, $u$ is regarded as particle density and represented in physical space by moving particles, while the other components are approximated in frequency space by Fourier series expansions. Our main objective here is to establish a new numerical method based on the original SIPF-PIC approach to the more complex system \eqref{eq:haptotaxis-system} by building reaction terms in the evolution of particle density.

To motivate our method, we first review mathematical models and numerical solutions for chemotaxis, a biological phenomenon concerning the movement of organisms (e.g., bacteria) in response to signals, typically chemical substances known as chemo-attractants, which can be produced by the organisms themselves. The Keller-Segel (KS) system\cite{KS} describes chemotaxis-driven aggregation processes, modeling the dynamics of organisms in response to the gradients of the chemo-attractants. The KS system demonstrates interdisciplinary utility, with applications spanning biological, ecological, and medical domains. The classical KS system on a bounded spatial domain $\Omega \subset \mathbb{R}^d$ can be written as:
\begin{equation}
\left\{
\begin{aligned}
\partial_t u &= \nabla \cdot \left( \mu \, \nabla u - \chi u \, \nabla \, v \right), \\
\epsilon \, \partial_t v &= \Delta v - k^2 \, v + u, \quad \mathbf{x} \in \Omega,\ t \in [0,T],
\end{aligned}
\right.
\label{eq:para_system}
\end{equation}
where $u(\mathbf{x},t)$ is the density of the particles, $v(\mathbf{x},t)$ is the concentration of the chemoattractant, $\chi, \mu >0$ are the chemotactic sensitivity and diffusivity, and $\epsilon, k\geq0$ are the relaxation and
degradation rates. The model is called elliptic if
$\epsilon = 0$ (when $c$ evolves rapidly to a local equilibrium), and parabolic if $\epsilon > 0$.

Numerical methods are commonly employed to study the solutions of the KS system, among which mesh-based methods are the most widely used approaches. Chertock et al.\cite{pos_preserving} proposed a high-order hybrid finite volume and finite difference scheme with positivity-preserving properties for two-dimensional (2D) problems. Shen et al.\cite{shen2020unconditionally} proposed an energy dissipation and bound-preserving scheme that is not restricted to specific spatial discretizations. Chen et al.\cite{chen2022error} developed a fully-discrete finite element scheme for the 2D parabolic-elliptic case, and demonstrated that the scheme exhibits finite-time blow-up under assumptions analogous to those of the continuous case. Other works include a positivity-preserving and asymptotic-preserving semi-discrete scheme\cite{LiuJian-Guo2018Paap}, and a saturation concentration setting\cite{hillen2001global} which prevents blow-up and yields spiky solutions. Beyond these 2D cases, Gnanasekaran et al.\cite{3Dblowup} simulated a 3D blow-up solution in a cube with a finite difference scheme in a short time.

In addition to the Eulerian discretization methods mentioned above, there have been notable advancements in the Lagrangian framework for the KS system \eqref{eq:para_system} and related equations. The parabolic-elliptic case is comparatively simpler, characterized by a memoryless particle interaction system\cite{velazquez2004point} with singularities arising when a pair of particles approach closely. Several particle methods seek to circumvent the singularity; Havskovéč and Ševčovič\cite{havskovec2009stochastic} employed a regularized interaction potential, and Liu et al.\cite{liu2017random} developed a random particle blob method with a mollified kernel. Convergence studies for these approaches have been conducted in\cite{havskovec2011convergence,mischler2013kac} and\cite{liu2019propagation}, with mean-field equations serving as a key tool\cite{mischler2013kac}. Other works include the computation of 2D advective parabolic-elliptic KS systems with passive flow\cite{khan2015global}, and the reduction of computational costs via the random batch method for particle interactions\cite{random_batch}.

Existing Lagrangian frameworks for parabolic-elliptic cases cannot be directly extended to fully parabolic systems, as the chemo-attractant concentration no longer achieves rapid local equilibration and the drift terms depend on historical particle positions\cite{chen2022mean,fournier2023particle}. Stevens\cite{stevens2000} developed an $N$-particle system for the fully parabolic case, representing the chemo-attractant concentration with particles subject to probabilistic creation and annihilation. However, this method also suffers from the accumulation of per-step complexity over time. To address this limitation, Wang et al.\cite{SIPF1,hu2025convergence} developed a stochastic interacting particle-field (SIPF) method that computes the chemo-attractant concentration through spectral decomposition and maintains only the current-step particle positions, which ensures that the per-step complexity remains constant over time. Hu et al\cite{hu2025fast} employed the Particle-in-Cell (PIC) technique to reduce the complexity of the SIPF method from $O(PH^d)$ to $O(P+H^d \log H)$, where \(P\) is the particle count and \(H\) is the number of Fourier modes in each dimension. We also refer the interested reader to the interacting particle method (IPM) developed by us for computing the principal eigenvalue of 3D or high-dimensional elliptic non-self-adjoint operators. The IPM can be used to compute KPP front speeds of RDA systems in chaotic and random flows\cite{lyu2022convergent,zhang2025convergent} and large deviation rate functions of entropy production for diffusion processes\cite{wu2025computing}, which demonstrates the advantages of the Lagrangian framework for solving complex PDE problems.

The algorithm to be developed in this paper for the RDA system (\ref{eq:haptotaxis-system}) is motivated by the ideas in\cite{hu2025fast,hu2025stochastic} that work efficiently in three space dimensions in computing nearly singular solutions. In system \eqref{eq:para_system}, only advection and diffusion are present, so particles of fixed mass can be used to represent and update cell density \(u\). In this paper, we extend the SIPF-PIC method to the case with reactions, handling the reaction term through variations in particle mass while maintaining the original $O\left(P + H^3 \log H\right)$ computational complexity per timestep, and refer to this extension as the SIPF-W method (Stochastic Interacting Particle-Field method with Weighted particles).

In addition to the inherent positivity preservation of cell density afforded by the empirical particle measures, the concentrations of chemical species reconstructed on the spatial grid are also unconditionally positivity-preserving, ensuring the physical realizability of the numerical solution at every stage of the computation. For a finite time \(T\), the rate of change of particle weights used to solve system \eqref{eq:haptotaxis-system} is bounded (see Lemma~\ref{lemma:mass}). Consequently, the ratio between the maximum and minimum particle weights remains controllable, allowing us to conduct a systematic and in-depth study on the error analysis for the proposed SIPF-W method.

The convergence of our method is theoretically guaranteed by the analysis in Section~\ref{sec:error}. Our main result, presented in Theorem \ref{thm:overall}, shows the convergence of the numerical solution \(\hat{u}\) obtained by the SIPF-W method to the reference solution \(u\) under certain non-blow-up assumptions (see Assumption \ref{as}) and regularization (see \eqref{eq:m_reg}). Specifically, while the particle count is large enough \((P \geq H^\frac{40}{13})\) and the computational time $T$ is fixed, the error in the particle distribution measured by the Wasserstein distance \( \mathcal{W}_1 \) scales as \(O\left(\tau + H^{-\frac{16}{13}}\right)\), where \(\tau\) is the timestep. For long-term simulations, however, the mass ratio becomes uncontrolled, and weight degeneracy may occur. We propose a residual resampling algorithm (Alg.~\ref{alg:res_res}) to restore the particle masses to the average value.

To complement the theoretical analysis, we also present numerical experiments to verify the convergence of SIPF-W. We follow the benchmark example described in\cite{walker2007global} and reproduce the particle patterns described therein. Since the system does not have an explicit solution, we verify that the numerical results obtained by our method are very close to those obtained by the mesh-based method mentioned in\cite{chertock2008second}. To the best of our knowledge, the numerical computation for the 3D problem has not been addressed in an earlier work. We extend the above benchmark example to 3D and verify the theoretical convergence order shown in Theorem \ref{thm:overall}.

The rest of the paper is organized as follows. Section \ref{subsec:SIPF_PIC} introduces the SIPF-W method, a hybrid particle-field method derived through bidirectional particle-grid coupling. Section \ref{subsec:PIC} introduces the particle-to-grid discretization and grid-to-particle interpolation operations required by the algorithm, details the algorithmic realization of the proposed method, and examines the computational cost. Section \ref{subsec:Properties} further derives structural properties, such as the interpolation error of PIC-related operations, an upper bound on the mass evolution rate, and the guaranteed positivity of the numerical concentration. Section \ref{sec:error} rigorously establishes a comprehensive error analysis framework that quantifies errors from various sources in the SIPF-W method. In Section \ref{sec:Num}, we first present a 2D numerical example (Section \ref{subsec:2D}) to validate the accuracy of our method by comparing with an established mesh-based algorithm, and then extend to a 3D case (Section \ref{subsec:3D}), where no prior work exists for direct comparison, to verify the theoretical convergence rate stated in Theorem \ref{thm:overall}. Finally, a summary of our findings and future research directions is given in Section \ref{sec:conclusion}.

\section{The Particle-Field Algorithm}
\label{sec:num_method}
In this section, we first introduce the Stochastic Interacting Particle Field method with particle-in-cell acceleration (SIPF-PIC), developed in\cite{SIPF1,hu2025convergence,hu2025fast}, which is a hybrid particle-spectral approach that uses particle-in-cell operations to ensure localized particle-frequency interactions. In this paper, we handle the reaction term using particles with varying mass instead of equal mass, and the algorithm is referred to as the SIPF-W method (Stochastic Interacting Particle-Field method with Weighted particles) in this case. Then, we introduce some algorithmic details and properties of the proposed SIPF-W method.

\subsection{SIPF-PIC Method}
\label{subsec:SIPF_PIC}
In this article, we assume the spatial dimension \(d=3\) by default. Throughout this section, the system is restricted to a domain \( \Omega = \left[-\frac{L}{2}, \frac{L}{2}\right]^3 \) with periodic boundary conditions.

In the SIPF-PIC method for solving system \eqref{eq:haptotaxis-system}, the diffusive-reactive chemical concentration \(\mathbf{g}(\mathbf{x},t)=(v(\mathbf{x},t);m(\mathbf{x},t);w(\mathbf{x},t))\) is represented in frequency space using Fourier series expansions, while the particle density \( u(\mathbf{x}, t) \) is represented in physical space by some moving particles. The temporal domain \([0, T]\) is partitioned into $N_T = \left\lceil \frac{T}{\tau} \right\rceil$ equal intervals of size \( \tau \), where the ceiling function $\lceil x \rceil$ denotes the smallest integer greater than or equal to $x$.

The frequency representation is used to enable fast convolution algorithms in the diffusion step. By the invertibility of the FFT, representing the concentration on an $H \times H \times H$ frequency grid is equivalent to working on a physical grid of the same size. At each time step \(t_n = n\tau \), the concentration is represented as:
\begin{equation}
\label{eq:vmw_hat}
\hat{\mathbf{g}}^{(n)}(\mathbf{x}):=\hat{\mathbf{g}}(\mathbf{x}, n\tau) = \sum_{\mathbf{q} \in U_H} \hat{\mathbf{k}}_{\mathbf{q}}^{(n)} e^{i \frac{2\pi}{L} \mathbf{q} \cdot \mathbf{x}},
\end{equation}
where \(\hat{\mathbf{k}}_{\mathbf{q}}^{(n)}=(\hat{k}_{v\mathbf{q}}^{(n)};\hat{k}_{m\mathbf{q}}^{(n)};\hat{k}_{w\mathbf{q}}^{(n)})\) are the time-dependent Fourier coefficients for \(\hat{\mathbf{g}}^{(n)}=(\hat{v}^{(n)};\hat{m}^{(n)};\hat{w}^{(n)})\), 
and the frequency index set \(U_H := \left\{-\frac{H}{2}, \ldots, \frac{H}{2}\right\}^3\). The particle density is represented as:
\begin{equation}
\label{eq:u_hat}
    \hat{u}^{(n)}(\mathbf{x}) = \sum_{p=1}^P \hat{a}_p^{(n)} \delta\left(\mathbf{x} - \widehat{\mathbf{X}}_p^{(n)}\right),
\end{equation}
where \(\delta(\cdot)\) is the Dirac delta function, \((\hat{a}_p^{(n)},\widehat{\mathbf{X}}_p^{(n)})\) represents the weight and the position of the $p$-th particle at time $t_n$.
\subsubsection{Updating chemical concentration \(\hat{\mathbf{g}}=(\hat{v};\hat{m};\hat{w})\)} 
The \((v,m,w)\)-equation in \eqref{eq:haptotaxis-system} can be written in the compact form:
\begin{equation}
\label{eq:g_t}
\mathbf{g}_t = \mathbf{D}\Delta \mathbf{g} - \mathbf{R} \mathbf{g} + \mathbf{S}(u,\mathbf{g}),
\end{equation}
where the diffusion coefficients and component-wise linear consumption rates \(\mathbf{D} =\\ \text{diag}(0,D_m,D_w),\mathbf{R} = \text{diag}(0,\beta,1)\) are diagonal matrices, the source term\\ \(\mathbf{S}(u,\mathbf{g})=(-\alpha m v;u;\gamma v - \eta (u) w)\).

Using the operator splitting method, we obtain a first-order  approximation for the time step $\tau$, where the reaction and diffusion steps are treated sequentially:
\begin{equation}
\label{eq:g_hat_phy}
\begin{aligned}
\hat{\mathbf{g}}^{(n+1)}&=\mathcal{D}_\tau (\mathcal{R}_\tau (\hat{\mathbf{g}}^{(n)},\hat{u}^{(n)}_H)),\\
\text{where} \quad \mathcal{D}_\tau(\mathbf{g})&= \mathbf{g} \ast \mathbf{G}_\tau, \quad \mathcal{R}_\tau (\mathbf{g},u) = \mathbf{g} + \tau \mathbf{S}(u,\mathbf{g}).
\end{aligned}
\end{equation}
The kernel of diffusion and linear consumption is defined as $\mathbf{G}_\tau(\mathbf{x}):=(\delta(\mathbf{x});\\(4\pi D_m \tau)^{-\frac{3}{2}}\cdot e^{-\frac{|\mathbf{x}|^2}{2D_m \tau}-\beta \tau};(4\pi D_w \tau)^{-\frac{3}{2}}\cdot e^{-\frac{|\mathbf{x}|^2}{2D_w \tau}-\tau})$, where $\delta(\mathbf{x})$ denotes the Dirac delta function at the origin, corresponding to the fact that the component $v$ does not undergo diffusion or linear consumption.

Applying the Fourier transform to \eqref{eq:g_hat_phy} yields:
\begin{equation}
\label{eq:k_q_update}
\hat{\mathbf{k}}_{\mathbf{q}}^{(n)} = \widehat{\mathbf{K}}_{\tau, \mathbf{q}}\left(\hat{\mathbf{k}}_{\mathbf{q}}^{(n-1)}+\tau\mathcal{F}_H[\mathbf{S}(\hat{u}^{(n-1)},\hat{\mathbf{g}}^{(n-1)})]_{\mathbf{q}} \right),
\end{equation}
where the Fourier operator \(\mathcal{F}[f]_\mathbf{q} = L^{-3} \int_{\mathbf{x}\in \Omega}e^{-\frac{2 \pi i}{L}\mathbf{q}\cdot \mathbf{x}}f(\mathbf{x})\mathrm{d}\mathbf{x}\), the theoretical kernel of diffusion and linear consumption in the frequency grid is defined as \(\mathbf{K}_{\tau,\mathbf{q}} := \\\exp\left(-\frac{4\pi^2|\mathbf{q}|^2}{L^2}\mathbf{D}\tau-\mathbf{R}\tau\right)\) with the exponential function applied component-wise.

To ensure the positivity-preserving property of the algorithm in the physical domain, we can construct the kernel \(\widehat{\mathbf{G}}_\tau\) directly on the physical grid and transform it to the frequency domain via the Discrete Fourier Transform (DFT) to obtain \(\widehat{\mathbf{K}}_\tau\). When the grid spacing $L/H$ tends to zero faster than the step size $\tau$ (i.e., $H^{-\psi}=o(\tau)$ for some \(\psi < 2\)), we can construct a kernel $\widehat{\mathbf{K}}_\tau$ whose inverse DFT yields a positivity-preserving quantity on the physical grid, and this kernel is exponentially close to the theoretical kernel $\mathbf{K}_\tau$, see Lemma~\ref{lemma:kernel}.

The nonlinear reaction term \(\mathbf{S}(\cdot,\cdot)\) is evaluated explicitly to avoid solving nonlinear systems. The particle-in-cell Fourier operator \(\mathcal{F}_H\) with spatial resolution \(H\) is claimed in \eqref{eq:def_fh}, which reduces the cost of particle-frequency interaction from \(O(P H^3)\) to \(O(P+H^3 \log H)\)\cite{hu2025fast}. Since \(\mathbf{D},\mathbf{R}\) are diagonal matrices, the update is performed component-wise.

\subsubsection{Updating particle density \(u\)}
The governing equation for \(u\) in \eqref{eq:haptotaxis-system} corresponds to the following SDE with an evolving weight \(a(t)\), where \(\mathbf{W}_{t} = \sqrt{2 D_u} \mathbf{B}_{t}\) and \(\mathbf{B}_{t}\) is a standard Brownian motion:
\begin{equation}
\label{eq:SDE}
\begin{cases}
\mathrm{d}\mathbf{X} = \chi \nabla v (\mathbf{X},t)\mathrm{d}t + \mathrm{d}\mathbf{W}_{t},\\
\mathrm{d}(\log a) = \rho(w(\mathbf{X},t))-1.
\end{cases}
\end{equation}
We apply the Euler-Maruyama scheme:
\begin{equation}
\begin{cases}
\widehat{\mathbf{X}}_p^{(n)} = \widehat{\mathbf{X}}_p^{(n-1)} + \chi \tau \nabla \hat{v}_H^{(n-1)}(\widehat{\mathbf{X}}_p^{(n-1)}) + \mathbf{W}_p^{(n)},\\
\hat{a}_p^{(n)}=\hat{a}_p^{(n-1)}(1+\tau \rho(\hat{w}_H^{(n-1)}(\widehat{\mathbf{X}}_p^{(n-1)})) - \tau),
\end{cases}
\label{eq:u_hat_update}
\end{equation}
where \(\mathbf{W}_p^{(n)} \sim \mathcal{N}(0,\, 2D_u\tau I_3) \) are independent 3D Wiener increments, \( I_3 \) is the \( 3 \times 3 \) identity matrix, and the drift \(\nabla \hat{v}^{(n-1)}(\widehat{\mathbf{X}}_p^{(n-1)})\) is derived from
\begin{equation}
\label{eq:gradv}
\nabla \hat{v}^{(n-1)}(\widehat{\mathbf{X}}_p^{(n-1)}) = \sum_{\mathbf{q} \in U_H}  \mathbf{q} \frac{2\pi i}{L} \hat{k}^{(n-1)}_{v\mathbf{q}} e^{\frac{2\pi i}{L} \mathbf{q} \cdot \widehat{\mathbf{X}}_p^{(n-1)}}.
\end{equation}
The particle-in-cell interpolated field \(\nabla \hat{v}_H\) and \(\hat{w}_H\) is defined in Section \ref{subsec:PIC}, which also reduces the cost of particle-frequency interaction from \(O(P H^3)\) to \(O(P+H^3 \log H)\)\cite{hu2025fast}.

\subsection{Particle-in-Cell Operations and Algorithmic Details}
\label{subsec:PIC}
In this section, we detail the algorithmic implementation of SIPF-W and analyze its computational complexity. Then we supplement some algorithmic properties, such as the upper bound on the rate of change of particle masses, and the positivity-preserving property of the concentrations maintained by the algorithm.

\subsubsection{Particle-to-Grid Transfer}
Recall that the concentration \(\hat{\mathbf{g}}^{(n)}(\mathbf{x})\) is represented by the frequency parts \(\hat{\mathbf{k}}_\mathbf{q}^{(n)}\); see \eqref{eq:vmw_hat}. The update \eqref{eq:k_q_update} of \(\hat{\mathbf{k}}_\mathbf{q}^{(n)}\) requires the source term \(\mathcal{F}[\mathbf{S}(u,{\mathbf{g}})]_{\mathbf{q}}\) in the frequency grid, which is equivalent to \(\mathbf{S}(u(\mathbf{x}),{\mathbf{g}(\mathbf{x})})\) on the physical grid by applying the Fast Fourier Transform (FFT). Notice that \(\hat{\mathbf{g}}^{(n)}(\mathbf{x})\) can be reconstructed onto the physical grid from \(\hat{\mathbf{k}}_\mathbf{q}^{(n)}\) directly by FFT. However, the particle density \(\hat{u}^{(n)}(\mathbf{x})\) is represented by moving particles \eqref{eq:u_hat}, and we need to transfer \(\hat{u}^{(n)}(\mathbf{x})\) onto the physical grid.

A conventional particle-in-cell (PIC) method\cite{alma99} estimates the density function at grid points \(\overline{u}^{(n)}\left(\mathbf{q}\frac{L}{H}\right)\) as a regularized Monte Carlo quadrature of the particle distribution \(\hat{u}^{(n)}(\mathbf{x})\):
\begin{equation}
\label{eq:def_u_overline}
\overline{u}^{(n)}\left(\mathbf{q}\frac{L}{H}\right) = \int_{\mathbf{x} \in \Omega} \hat{u}^{(n)}(\mathbf{x}) R_h\left(\mathbf{x},\mathbf{q}h\right) \mathrm{d}\mathbf{x} = \sum^P_{p=1} \hat{a}_p\hat{u}^{(n)}(\mathbf{X}_p) R_h\left(\mathbf{X}_p,\mathbf{q}h\right),
\end{equation}
where \(R_h\left(\mathbf{x},\mathbf{q}h\right)\) is the assignment function with support scale \(h=\frac{L}{H}\), satisfying the partition of unity \(\int_{\mathbf{x} \in \Omega} R_h\left(\mathbf{x},\mathbf{q}h\right) \mathrm{d}\mathbf{x}=1\) for all \(\mathbf{q} \in U_H\). Furthermore, the assignment function \(R_h\) is locally supported, meaning that each particle \(\mathbf{X}_p\) interacts with only a constant number of grid points.

For example, in a trilinear interpolation with a second-order truncation error, the assignment function is represented as
\begin{equation}
R_h\left(\mathbf{x}, \mathbf{q}h\right) = h^{-3} \prod_{j=1}^3\max\left(1-h^{-1}\left|x_j - q_j h\right|,0\right).
\end{equation}
In particular, $R_h(\mathbf{x},\cdot)$ is nonzero only for the \(2^3=8\) vertices of the grid cell containing the coordinate $\mathbf{x}$.

To achieve higher precision, we also introduce assignment functions with higher-order truncation errors. As proposed by\cite{hu2025fast}, an assignment function with a fourth-order truncation error is constructed as follows, where the distance is measured by grid length on the \(j\)-th axis \(p_j=h^{-1}|x_j - q_j h|\) for \(1\leq j \leq 3\):
\begin{equation}
\begin{aligned}
    &R_h^{(4)}\!\left(\mathbf{x}, \mathbf{q}_h\right) = \\
    &\qquad
    \begin{cases}
        h^{-3} \left(\prod_{j=1}^3 (1-p_j)\right)
        \left(1+\sum_{j=1}^3 \frac{p_j(1-p_j)}{2}\right), 
        & \text{if } p_1,p_2,p_3 \leq 1; \\[2pt]
        -\frac{h^{-3}}{6}(1-p_2)(1-p_3)(p_1-1)(2-p_1)(3-p_1),  
        & \text{if } 1 < p_1 \leq 2 \text{ and } p_2,p_3 \leq 1;\\[2pt]
        -\frac{h^{-3}}{6}(1-p_1)(1-p_3)(p_2-1)(2-p_2)(3-p_2),  
        & \text{if } 1 < p_2 \leq 2 \text{ and } p_1,p_3 \leq 1;\\[2pt]
        -\frac{h^{-3}}{6}(1-p_1)(1-p_2)(p_3-1)(2-p_3)(3-p_3),  
        & \text{if } 1 < p_3 \leq 2 \text{ and } p_1,p_2 \leq 1;\\[2pt]
        0, & \text{other cases}.
    \end{cases}
\end{aligned}
\end{equation}

We transfer the empirical measure \(\hat{u}^{(n)}(\mathbf{x})\) onto the physical grid using the following algorithm (Algorithm \ref{alg:p2g}). The notation \(\Gamma_{p,H}:=\left\{\mathbf{q}\in U_H:R_h(\mathbf{X}_p,\mathbf{q}h)>0\right\}\) denotes the set of indices of the grids interacting with \(\mathbf{X}_p\). To ensure theoretical convergence of the algorithm, we apply a low-pass filter\cite{hu2025fast} in the frequency domain $K_\mathbf{q} = \exp( -2 \pi^2 |\mathbf{q}|^2H_0^{-2})$ for some \(H_0 \leq H\), which is equivalent to convolving the density function with a Gaussian kernel of variance $L^2 H_{0}^{-2}$.
\begin{algorithm}
\caption{Particle-to-Grid Transfer}
\label{alg:p2g}
\begin{algorithmic}[1]
 \State Input coordinates \(\{\mathbf{X}_p\}_{1\leq p \leq P}\), weights \(\{\hat{a}_p\}_{1\leq p \leq P}\)
    \State Initialize \(\overline{u}(\mathbf{q}h) \leftarrow 0\) for all \(\mathbf{q} \in U_H\)
    \For{\(p = 1\) \textbf{to} \(P\)}
        \State Find \(\Gamma_{p,H}\)
        \For{\(\mathbf{q} \in \Gamma_{p,H}\)}
            \State Update contribution: \(\overline{u}(\mathbf{q}h) \leftarrow \overline{u}(\mathbf{q}h) + \hat{a}_p R_h(\mathbf{X}_p,\mathbf{q}h)\)
        \EndFor
    \EndFor
    \State Output \(\overline{u}(\mathbf{q}h)\) for all \(\mathbf{q} \in U_H\)
\end{algorithmic}
\end{algorithm}

The numerical frequency source term required in \eqref{eq:k_q_update} is defined as
\begin{equation}
\label{eq:def_fh}
\mathcal{F}_H[\mathbf{S}(\hat{u}^{(n-1)},\hat{\mathbf{g}}^{(n-1)})]_{\mathbf{q}}:=\mathcal{F}[\mathbf{S}(\overline{u}^{(n-1)},\hat{\mathbf{g}}^{(n-1)})]_{\mathbf{q}} \quad \text{for } \mathbf{q} \in U_H,
\end{equation}
where \(\overline{u}^{(n-1)}\) is defined in \eqref{eq:def_u_overline} and numerically constructed in Algorithm \ref{alg:p2g}.
\subsubsection{Grid-to-Particle Interpolation}
The update \eqref{eq:u_hat_update} of the particle positions and weights \((\mathbf{X}_p^{(n)},\hat{a}_p^{(n)})\) requires the values of \(\nabla \hat{v}^{(n-1)}\) and \(\hat{w}^{(n-1)}\) at \(\mathbf{X}_p^{(n-1)}\). Since the fields \(\nabla \hat{v}^{(n-1)}\) and \(\hat{w}^{(n-1)}\) are constructed only at grid points, we need to interpolate them at an arbitrary position \(\mathbf{X} \in \Omega\).

Consider a smooth target function \( f: \mathbb{R}^3 \to \mathbb{R} \). The PIC method\cite{alma99} represents \(f(\mathbf{X})\) as a combination of the function values at grid points:
\begin{equation}
\label{eq:g2p}
\overline{f}(\mathbf{X}) = \sum_{\mathbf{q} \in U_H} K_h\left(\mathbf{X}, \mathbf{q}h\right)f\left(\mathbf{q}h\right),
\end{equation}
where the particle-grid interaction kernel $K_h\left(\mathbf{X}, \mathbf{q}h\right)$ satisfies \(\sum_{\mathbf{q}\in U_H}K\left(\mathbf{X}, \mathbf{q}h\right)=1\) for every $\mathbf{X} \in \Omega$.

The interaction kernel \(K_h\) is also locally supported. Furthermore, \(K_h\) corresponds to the assignment function \(R_h\) in \eqref{eq:def_u_overline} with the relationship \(K_h = h^3 R_h\), and they share the same order of truncation error.

We interpolate a function \(f\) at some arbitrary positions \(\mathbf{X}_1,...,\mathbf{X}_p \in \Omega\) using the following algorithm (Algorithm \ref{alg:g2p}). Since the interaction kernel \(K_h\) is equivalent to the assignment function \(R_h\), \(K_h\) has the same support set \(\Gamma_{p,H}\) for the \(p\)-th particle.
\begin{algorithm}
\caption{Grid-to-Particle Interpolation}
\label{alg:g2p}
\begin{algorithmic}[1]
    \State Input \(f(\mathbf{q}h)\) defined on grid indices \(\mathbf{q} \in U_H\), coordinates \(\{\mathbf{X}_p\}_{1\leq p \leq P}\)
    \For{\(p = 1\) \textbf{to} \(P\)}
        \State Initialize : \(f_p \leftarrow 0\)
        \State Find \(\Gamma_{p,H}\)
        \For{\(\mathbf{q} \in \Gamma_{p,H}\)} 
            \State Accumulate contribution: \(f_p \leftarrow f_p + K_h\left(\mathbf{X}_p, \mathbf{q}h\right)f\left(\mathbf{q}h\right)\)
        \EndFor
        \State Output interpolated value \(\overline{f}(\mathbf{X}_p) = f_p\)
    \EndFor   
\end{algorithmic}
\end{algorithm}

The interpolated values required in \eqref{eq:u_hat_update} are defined as
\begin{equation}
\label{eq:g2ph}
\nabla \hat{v}_H(\mathbf{X}) = \sum_{\mathbf{q} \in U_H} K_h\left(\mathbf{X}, \mathbf{q}h\right)\nabla \hat{v}\left(\mathbf{q}h\right), \hat{w}_H(\mathbf{X}) = \sum_{\mathbf{q} \in U_H} K_h\left(\mathbf{X}, \mathbf{q}h\right)\hat{w}\left(\mathbf{q}h\right),
\end{equation}
where \(h=\frac{L}{H}\) is the grid length.

\subsubsection{Overall Algorithm and Complexity}
We summarize the SIPF-W method as the following Algorithm \ref{alg:sipf-pic}. The complexity of the algorithm is stated in the following lemma.

\begin{algorithm}
\begin{algorithmic}[1]
\caption{SIPF-W method}
\label{alg:sipf-pic}
\Statex \textbf{Initialization:}
\State Sample i.i.d \(\{\widehat{\mathbf{X}}_p^{(0)}\}_{1\leq p \leq P}\sim u(\cdot,0)\)
\State Compute \(\hat{\mathbf{k}}_\mathbf{q}^{(0)}=\mathcal{F}[\mathbf{g}(\cdot,0)]_\mathbf{q}\)
\Statex
\For {$n = 1, 2, \dots,N_T$}
\Statex \textbf{Update chemical concentration} \(\mathbf{g}\):
\State Compute \(\overline{u}^{(n-1)}(\mathbf{q}h)\) at grid points by Alg. \ref{alg:p2g}
\State Compute \(\mathcal{F}_H[\mathbf{S}(\hat{u}^{(n-1)},\hat{\mathbf{g}}^{(n-1)})]_{\mathbf{q}}\) via FFT
\State Update frequency components:\\ \quad \quad
\(\hat{\mathbf{k}}_{\mathbf{q}}^{(n)} \leftarrow \widehat{\mathbf{K}}_\tau \left(\hat{\mathbf{k}}_{\mathbf{q}}^{(n-1)}+\tau\mathcal{F}_H[\mathbf{S}(\hat{u}^{(n-1)},\hat{\mathbf{g}}^{(n-1)})]_{\mathbf{q}} \right)\)
\State Apply low-pass filter: \(\hat{\mathbf{k}}_{\mathbf{q}}^{(n)} \leftarrow \hat{\mathbf{k}}_{\mathbf{q}}^{(n)}\exp( -2 \pi^2 |\mathbf{q}|^2H_0^{-2})\)
\Statex \textbf{Update particle density} \(u\):
\State Compute \(\nabla \hat{v}(\mathbf{q}h), \hat{w}(\mathbf{q}h)\) at grid points via FFT
\For{$p=1$ to $P$}
\State Compute \(\nabla \hat{v}_H(\widehat{\mathbf{X}}_p^{(n-1)}), \hat{w}_H(\widehat{\mathbf{X}}_p^{(n-1)})\) by Alg. \ref{alg:g2p}
\State Sample \(\mathbf{W}_p^{(n)}\sim \mathcal{N}(0, 2D_u \tau \mathbf{I})\)
\State Update positions and weights:\\ \quad \quad \qquad
\(\widehat{\mathbf{X}}_p^{(n)} = \widehat{\mathbf{X}}_p^{(n-1)} + \chi \tau \nabla \hat{v}_H^{(n-1)}(\widehat{\mathbf{X}}_p^{(n-1)}) + \mathbf{W}_p^{(n)},\)\\\quad \quad \qquad\(
\hat{a}_p^{(n)}=\hat{a}_p^{(n-1)}(1+\tau \rho(\hat{w}_H^{(n-1)}(\widehat{\mathbf{X}}_p^{(n-1)})) - \tau)\)
\EndFor
\EndFor
\end{algorithmic}
\end{algorithm}

\begin{theorem}
\label{thm:overall_complexity}
The single-step complexity of Algorithm~\ref{alg:sipf-pic} is \( O \left(H^3 \log H + P \right) \).
\end{theorem}
\begin{proof}
 Each iteration costs \( O(H^3 \log H + P) \): spectral updates require FFT with complexity \( O(H^3 \log H) \) and particle-to-grid transfer with \( O(P) \), while particle updates require inverse FFT with \( O(H^3 \log H) \) and grid-to-particle interpolation with \( O(P) \).
\end{proof}

\subsection{Algorithmic Properties}
\label{subsec:Properties}
In this section, we further derive several structural properties, including the interpolation error of PIC-related operations, an upper bound on the mass evolution rate, the resampling operation in long-term simulations, and the guaranteed positivity of the numerical concentration.

\subsubsection{Interpolation error of PIC-related operations}
For the overall error analysis of Algorithm \ref{alg:sipf-pic} in Section \ref{sec:error}, we present the error bounds of the particle-in-cell operations as follows. The detailed proof can be found in\cite{hu2025fast}.
\begin{theorem}
\label{thm:pic_error}
With the fourth-order interpolation scheme, the error of the numerical frequency source term \(\mathcal{F}_H\) satisfies:
\begin{equation}
\left|\mathcal{F}_H[\mathbf{S}(\hat{u},\hat{\mathbf{g}})]_{\mathbf{q}}-\mathcal{F}[\mathbf{S}(\hat{u},\hat{\mathbf{g}})]_{\mathbf{q}}\right| \leq C_4 L^{-3}\left|\mathbf{q}\frac{2\pi}{L}\right|^4 \left(\frac{L}{H}\right)^4
\end{equation}
for some \(C_4>0\) independent of \(L\) and \(H\).

Assume \(w,\nabla v \in C^2(\Omega)\). With the second-order interpolation scheme, the error of the interpolated gradient field \(\nabla v_H\) and the interpolated concentration \(w_H\) satisfies the following for any \(\mathbf{X} \in \Omega\):
\begin{equation}
\begin{aligned}
\left|\nabla v_H(\mathbf{X})-\nabla v(\mathbf{X})\right|&\leq K_2 \left(\frac{L}{H}\right)^2,\\
\left|w_H(\mathbf{X})-w(\mathbf{X})\right|&\leq K_2 \left(\frac{L}{H}\right)^2
\end{aligned}    
\end{equation}
for some \(K_2>0\) independent of \(L\) and \(H\).
\end{theorem}

\subsubsection{Resampling and Long-term Simulation}
Concerning the mass variation of the particles, we have the following lemma:
\begin{lemma}
\label{lemma:mass}
The following estimate holds for every particle \(1\leq p \leq P\) and timestep \(n\): \begin{equation}
\begin{aligned}
(1-\tau)\hat{a}_p^{(n)} \leq \hat{a}_p^{(n+1)} \leq (1+\tau) \hat{a}_p^{(n)}.
\end{aligned}
\end{equation}   
\end{lemma}
\begin{proof}
This follows directly from the update rule \eqref{eq:u_hat_update} and the fact that \(-1 \leq \rho(w)-1  \leq 1\) where \(\rho(w)=\frac{2w}{1+w}\) and \(w\geq 0\).
\end{proof}

Summing over all particles \(1\leq p \leq P\) yields the following corollary.
\begin{corollary}
\label{corollary:mass}
The total masses \(\widehat{M}^{(n)}:=\sum_{p=1}^P \hat{a}_p^{(n)}\) satisfy:
\begin{equation}
M_0e^{-t_n} \leq \widehat{M}^{(n)} \leq M_0e^{t_n},
\end{equation}
where \(M_0:=\int_\Omega u(\mathbf{x},0)\,\mathrm{d}\mathbf{x}\) is the total initial mass and \(t_n=n\tau\).
\end{corollary}
By Lemma \ref{lemma:mass}, the ratio of maximum to minimum particle mass is bounded over finite time \(t=n\tau \):
\begin{equation}
\frac{\max_p\hat{a}_p^{(n)}}{\min_p\hat{a}_p^{(n)}} \leq e^{2t}.
\end{equation}
Therefore, the statistical law of large numbers can be applied (see Section \ref{sec:error}).

For long-term simulations, the mass ratio becomes uncontrolled, leading to severe weight degeneracy where a few particles dominate the ensemble while the rest contribute negligibly; thus, time-periodic resampling is required in long-time simulations. To perform this step, the following residual resampling algorithm (Alg.~\ref{alg:res_res}) can be applied. This scheme operates by replicating particles according to the integer part of their normalized mass and stochastically retaining the fractional remainder.

\begin{algorithm}
\caption{Residual Resampling}
\label{alg:res_res}
\begin{algorithmic}[1]
\State Input step count \(n\), particles \(\{\widehat{\mathbf{X}}_p^{(n)}\}_{1\leq p \leq P}\), weights \(\{\hat{a}_p^{(n)}\}_{1\leq p \leq P}\)
\Statex
\State \textbf{Initialize:} empty list of new particles \(\mathcal{X}^{(n)} \gets \emptyset\)
\State Select threshold weight \(a_0 = \frac{1}{P} \sum_{p=1}^P \hat{a}_p^{(n)}\)
\Statex
\For{$p = 1$ to $P$}
    \State \(k_p \gets \lfloor \hat{a}_p^{(n)} / a_0 \rfloor\) \Comment{Integer part: deterministic splitting}
    \State Add \(k_p\) copies of \(\widehat{\mathbf{X}}_p^{(n)}\) to \(\mathcal{X}^{(n)}\)
    \State \(r_p \gets \hat{a}_p^{(n)} - k_p \cdot a_0\) \Comment{Fractional residual mass}
\EndFor
\Statex
\State Let \(N_{\text{stoch}} \gets P - \sum_{p=1}^P k_p\)
\State Sample \(N_{\text{stoch}}\) indices from \(\{1,\dots,P\}\) with probabilities proportional to \(r_p\)
\For{each sampled index \(p\)}
    \State Add one copy of \(\widehat{\mathbf{X}}_p^{(n)}\) to \(\mathcal{X}^{(n)}\)
\EndFor
\Statex
\State \textbf{Output:} Resampled list of particles \(\mathcal{X}^{(n)}\)
\end{algorithmic}    
\end{algorithm}

The resampled list of particles \(\mathcal{X}^{(n)}\) in Alg.~\ref{alg:res_res} consists of \(P\) equal-weighted particles, which eliminates the possible weight degeneracy and restores the validity of the statistical law of large numbers.

\subsubsection{Positivity-preserving Properties}

We present the following lemmas to show that in addition to the naturally preserved positivity of $u$ represented by particles, the concentrations $v,m,w$ also preserve positivity in the physical domain. The first lemma shows that when the grid spacing $L/H$ tends to zero faster than the step size $\tau$, we can construct a kernel $\widehat{\mathbf{K}}_\tau$ whose inverse DFT yields a positivity-preserving quantity on the physical grid, and this kernel is exponentially close to the theoretical kernel $\mathbf{K}_\tau$. Similar exponential tail bounds for Fourier-grid approximations of Gaussian kernels have been derived by Barnett, Greengard, and Rachh\cite{uniform}, where the truncation error of the squared-exponential kernel is shown to decay exponentially with the frequency cut-off. Our theorem provides a constructive application based on this truncation error estimate.
\begin{lemma}
\label{lemma:kernel}
Given the grid spacing $L/H \to 0$ sufficiently fast relative to $\tau$, there exists $\widehat{K}_{m\tau} \in \mathbb{R}^{U_H}$ such that all components of the discrete inverse Fourier transform $\mathcal{F}_H^{-1} \widehat{K}_{m\tau}$ are positive, and for each $\mathbf{q} \in U$,
\[
\left|\widehat{K}_{m\tau,\mathbf{q}} - K_{m\tau,\mathbf{q}}\right| < 4e^{-\pi^2 D_m \tau H^2 L^{-2}},
\]
where $D_m$ is the diffusion coefficient of the concentration \(m\).
\end{lemma}
\begin{proof}
For each $\mathbf{q} \in U$, let $\widehat{K}_{m\tau,\mathbf{q}} = \sum_{\mathbf{n} \in \mathbb{Z}^3} K_{m\tau,\mathbf{q} + \mathbf{n}H}$. The inverse DFT of this aliased kernel $\widehat{G}_{m\tau}$ gives the sum over all periodic copies of the Gaussian kernel in the physical domain; specifically, for every lattice point $\mathbf{x} = \mathbf{p} \frac{L}{H}$,
\begin{equation}
\widehat{G}_{m\tau}\left(\mathbf{p} \frac{L}{H}\right) = \sum_{\mathbf{n} \in \mathbb{Z}^3} G_{m\tau}\left(\mathbf{p}\frac{L}{H} + \mathbf{n}L\right).    
\end{equation}
Hence, all its components are positive.

To estimate the difference between the aliased kernel $\widehat{K}_{m\tau,\mathbf{q}}$ and the principal component $K_{m\tau,\mathbf{q}}$, recall that
\(
K_{m\tau,\mathbf{q}} = e^{-\frac{4\pi^2 |\mathbf{q}|^2}{L^2} D_m \tau - \beta \tau}.
\)
Summing over all non-principal components ($\mathbf{n} \neq \mathbf{0}$) gives
\begin{equation}
\left|\widehat{K}_{m\tau,\mathbf{q}} - K_{m\tau,\mathbf{q}}\right| \leq \sum_{\mathbf{n} \in \mathbb{Z}^3 \setminus \{\mathbf{0}\}} e^{-\frac{4\pi^2 |\mathbf{q} + \mathbf{n}H|^2}{L^2} D_m \tau - \beta \tau} \leq \sum_{\mathbf{n} \in \mathbb{Z}^3 \setminus \{\mathbf{0}\}} e^{-\frac{\pi^2 H^2|\mathbf{n}|^2}{L^2} D_m \tau},
\end{equation}
which can be bounded and simplified to obtain the desired estimate.
\end{proof}
Since the remaining components of the error analysis (Section \ref{sec:error}) involve only polynomial convergence rates, and our parameter choice \(\tau = H^{-\frac{16}{13}}\) (Corollary \ref{cor:optimized_error}) satisfies the condition for exponential decay, the difference between the frequency-truncated kernel and the kernel preserving physical positivity becomes negligible. For the concentration $w$ that involves diffusion, we have a similar conclusion. Since the concentration $v$ only involves reaction terms, we can update $v$ implicitly:
\begin{equation}
\hat{v}^{(n+1)}=(1+\alpha \tau \hat{m}^{(n)})^{-1} \cdot \hat{v}^{(n)}.
\end{equation}
The following lemma shows that in this setting, Algorithm~\ref{alg:sipf-pic} is unconditionally positivity-preserving.
\begin{lemma}
For timestep \(\tau \leq 0.5\), the algorithm-maintained \(\hat{v}^{(n)}(\mathbf{q}\frac{L}{H}),\\\hat{m}^{(n)}(\mathbf{q}\frac{L}{H}),\hat{w}^{(n)}(\mathbf{q}\frac{L}{H}) \geq 0\) for any step count \(n \geq 0\) and grid location $\mathbf{q}\frac{L}{H}$.
\end{lemma}
\begin{proof}
By the update rule \eqref{eq:g_hat_phy} where the convolution kernel \(\mathbf{G}_\tau>0\), we only need to show $\mathcal{R}_\tau (v;m;w,u) = (v(1+\alpha \tau m)^{-1};(4\pi D_m \tau)^{-\frac{3}{2}}\cdot e^{-\frac{|\mathbf{x}|^2}{2D_m \tau}-\beta \tau} (m+\tau u);\\(4\pi D_w \tau)^{-\frac{3}{2}}\cdot e^{-\frac{|\mathbf{x}|^2}{2D_w \tau}-\tau}(w + \tau \gamma v - \tau \eta (u) w))\geq0$.

Since $\eta(u) = \frac{2u}{1+u} \leq 2$, we have $w + \tau \gamma v - \tau \eta (u) w \geq 0$ for $\tau \leq 0.5$. Hence, the updated $\hat{\mathbf{g}}^{(n+1)}$ is the convolution of the positive quantity $\mathcal{R}_\tau$ with the positive convolution kernel \(\mathbf{G}_\tau\), and is therefore also positive.
\end{proof}

\section{Convergence Analysis}
\label{sec:error}
In this section, we present the convergence result of the proposed SIPF-W method (Algorithm \ref{alg:sipf-pic}) for solving system \eqref{eq:haptotaxis-system} with a finite simulation time \(T\). The proof sketch shows that the particles maintained by our algorithm, and the independent branches of the particle updated according to the mean-field dynamics, remain very close to each other over a finite time.

\subsection{Preliminaries and Assumptions}
We introduce the following Brownian coupling technique to construct auxiliary random branches $(\widetilde{\mathbf{X}}_p^{(n)},\tilde{a}_p^{(n)})$ that bridge the ground truth particle density $u(\mathbf{x},n\tau)$ and the empirical particle density $\hat{u}^{(n)}(\mathbf{x})$ maintained in the algorithm.
\begin{definition}
\label{def:br_cp}
The auxiliary particles \( \widetilde{\mathbf{X}}_p^{(n)} \) and the corresponding weights \(\tilde{a}_p^{(n)}\) (for \( 1 \leq p \leq P \)) satisfy the following update process:
\begin{equation}
\label{eq:u_tilde_update}
\begin{aligned}
\widetilde{\mathbf{X}}_p^{(n)} &= \widetilde{\mathbf{X}}_p^{(n-1)} + \chi \tau \nabla v(\widetilde{\mathbf{X}}_p^{(n-1)},(n-1)\tau) + \mathbf{W}_p^{(n)},\\
\tilde{a}_p^{(n)} &=\tilde{a}_p^{(n-1)}(1+\tau \rho(w(\widetilde{\mathbf{X}}_p^{(n-1)},(n-1)\tau)) - \tau),
\end{aligned}
\end{equation}
with the initial distribution \(\widetilde{\mathbf{X}}_p^{(0)} \sim \text{i.i.d.}~ \mathbf{X}(0)\) where \(\text{Law}(\mathbf{X}(0))=u(\cdot,0)/M_0\) and \(M_0=\int_\Omega u(\mathbf{x},0)\,\mathrm{d}\mathbf{x}\), the initial masses \(\tilde{a}_p^{(0)}=M_0/P\), the Brownian motions \(\mathbf{W}_p^{(n)}\) are the same as in Algorithm \ref{alg:sipf-pic}.

The auxiliary discrete density \( \tilde{u}^{(n)} \) is defined as 
\begin{equation}
\tilde{u}^{(n)}(\mathbf{x}) = \sum\limits_{p=1}^{P} \tilde{a}_p^{(n)} \delta(\mathbf{x} - \widetilde{\mathbf{X}}_p^{(n)}).
\end{equation}
Furthermore, the auxiliary concentration $\tilde{\mathbf{k}}_\mathbf{q}^{(n)}$ in the frequency domain and $\tilde{\mathbf{g}}^{(n)}(\mathbf{x})$ in the physical domain are updated similarly to \eqref{eq:k_q_update}. The auxiliary concentration does not directly drive the movement of the auxiliary particles; it is only used for comparison with the concentration $\hat{\mathbf{g}}^{(n)}(\mathbf{x})$ maintained by the algorithm.
\begin{equation}
\label{eq:k_q_tilde_update}
\begin{aligned}
\tilde{\mathbf{k}}_{\mathbf{q}}^{(n)} &= \left(\mathbf{I} + \tau\left(\frac{4\pi^2|\mathbf{q}|^2}{L^2}\mathbf{D}+\mathbf{R}\right)\right)^{-1}\left(\tilde{\mathbf{k}}_{\mathbf{q}}^{(n-1)}+\tau\mathcal{F}[\mathbf{S}(\tilde{u}^{(n-1)},
\tilde{\mathbf{g}}^{(n-1)})]_{\mathbf{q}} \right),\\
\tilde{\mathbf{g}}^{(n)}&(\mathbf{x}) = \sum_{\mathbf{q} \in U_{H_0}} \tilde{\mathbf{k}}_{\mathbf{q}}^{(n)} e^{i \frac{2\pi}{L} \mathbf{q} \cdot \mathbf{x}}.
\end{aligned}
\end{equation}
The weighted particles \((\widetilde{\mathbf{X}}_p^{(n)}, \tilde{a}_p^{(n)})\) are distributed independently and identically according to a mass distribution function \(\phi^{(n)}(\mathbf{x},a)\); that is, every particle in the process \eqref{eq:u_tilde_update} has a distribution \(u^{(n)}(\mathbf{x})=\int a\phi^{(n)}(\mathbf{x},a)\,\mathrm{d}a\). The temporal semi-discrete implicit reference concentration $\mathbf{k}_\mathbf{q}^{(n)}$ (in the frequency domain) and $\mathbf{g}^{(n)}(\mathbf{x})$ (in the physical domain) are updated using \(u^{(n)}\):
\begin{equation}
\label{eq:k_q_ref_update}
\begin{aligned}
{\mathbf{k}}_{\mathbf{q}}^{(n)} &= \left(\mathbf{I} + \tau\left(\frac{4\pi^2|\mathbf{q}|^2}{L^2}\mathbf{D}+\mathbf{R}\right)\right)^{-1}\left({\mathbf{k}}_{\mathbf{q}}^{(n-1)}+\tau\mathcal{F}[\mathbf{S}({u}^{(n-1)},
{\mathbf{g}}^{(n-1)})]_{\mathbf{q}} \right),\\
{\mathbf{g}}^{(n)}&(\mathbf{x}) = \sum_{\mathbf{q} \in U_{H_0}} {\mathbf{k}}_{\mathbf{q}}^{(n)} e^{i \frac{2\pi}{L} \mathbf{q} \cdot \mathbf{x}}.
\end{aligned}
\end{equation}
\end{definition}

We impose the following assumptions to ensure the convergence of Algorithm~\ref{alg:sipf-pic}.
\begin{assumption}
\label{as}
The solution \(u,\mathbf{g}\) to \eqref{eq:haptotaxis-system} is periodic, that is, for \(q_1,q_2,q_3 \in \mathbf{Z}\),
\begin{equation}
\begin{aligned}
u(x_1+q_1 L,x_2+q_2 L,x_3 + q_3 L) &= u(x_1,x_2,x_3), \\
\mathbf{g}(x_1+q_1 L,x_2+q_2 L,x_3 + q_3 L)    &= \mathbf{g}(x_1,x_2,x_3).
\end{aligned}
\label{eq:periodic}
\end{equation}
The solution \(u\) is bounded, i.e., 
\begin{equation}
\begin{aligned}
|u(\mathbf{x},t)| \leq N_0 \quad \forall t \in \left[0,T\right], \mathbf{x} \in \Omega.
\end{aligned}
\end{equation}
The gradient of solution \(\nabla u\) is square-integrable, i.e.,
\begin{equation}
\begin{aligned}
\int_{\mathbf{x}\in \Omega}|\nabla u(\mathbf{x},t)|^2\,\mathrm{d}\mathbf{x} \leq N_{1,2} \quad \forall t \in \left[0,T\right], \mathbf{x} \in \Omega.
\end{aligned}
\end{equation}
The first to third order spatial derivatives of \(\mathbf{g}\) are bounded, i.e., for any spatial multi-index \( \mathbf{k} \) such that \( |\mathbf{k}| = \eta \) where \(\eta=1,2,3\), it holds that
\begin{equation}
\label{eq:c_smoothness}
|D_{\mathbf{k}} \mathbf{g}(\mathbf{x},t)| \leq M_\eta \quad \forall \mathbf{x} \in \left[-\frac{L}{2}, \frac{L}{2}\right]^3,t \in \left[0,T\right].
\end{equation}
The temporal derivative of \(\nabla \mathbf{g}\) is bounded, that is,
\begin{equation}
\left| \frac{\partial \nabla \mathbf{g}(\mathbf{x},t)}{\partial t} \right| \leq J_1\quad \forall \mathbf{x} \in \left[-\frac{L}{2}, \frac{L}{2}\right]^3,t \in \left[0,T\right].
\end{equation}
\end{assumption}

\begin{remark}
Since the vector $\mathbf{g} = (v; m; w)$, the smoothness conditions of \(\mathbf{g}\) listed in Assumption \ref{as} apply to each individual component \(v,m,w\).
\end{remark}

\subsection{Main Convergence Results}
In this section, we first present a series of lemmas that establish bounds for individual components of the error dynamics, and then prove Theorem \ref{thm:overall}, which constitutes the main convergence result of Algorithm \ref{alg:sipf-pic} for solving system \eqref{eq:haptotaxis-system}.

To avoid potential singularities, we regularize the governing equation for $m$ as follows:
\begin{equation}
\label{eq:m_reg}
m_t = D_m \Delta m - \beta m + K_\sigma \ast u,
\end{equation}
where \(K_\sigma\) is the convolution kernel corresponding to the three-dimensional normal distribution \(N(0,\sigma I_3)\) for some $\sigma \ll 1$.

The temporal error of the Euler--Maruyama semi-discretization is summarized below; see\cite{milstein2004stochastic} for standard strong-convergence results.
\begin{lemma}
\label{lemma:temporal}
Let \(\{(\mathbf{X}_p,a_p)\}_{1 \leq p \leq P}\) be \(P\) independent realizations of SDE \eqref{eq:SDE} with \(\mathbf{X}_p(0)=\widetilde{\mathbf{X}}_p^{(0)},a_p(0)=\tilde{a}_p^{(0)}\), and the Brownian increments \(\mathbf{W}_p\) satisfy
\begin{equation}
\mathbf{W}_p(n\tau) - \mathbf{W}_p((n-1)\tau) = \mathbf{W}_p^{(n)}. 
\end{equation}
The Euler-Maruyama semi-discretization exhibits first-order strong convergence for additive noise, i.e., when the diffusion term is independent of the solution process. More precisely, for \(t = n\tau \leq T\) with integer \(n\geq 1\) and any \(1 \leq p \leq P\),
\begin{equation}
\begin{aligned}
\bigl( \mathbb{E}\bigl[ | \mathbf{X}_p(t) - \widetilde{\mathbf{X}}_p^{(n)} |^2 \bigr] \bigr)^{\frac{1}{2}} &\leq e^{C_0 t} \tau,\\
\left(\mathbb{E}\bigl[ |\log a_p(t) - \log \widetilde{a}_p^{(n)} |^2 \bigr]\right)^\frac{1}{2}&\leq e^{C_0 t} \tau,
\end{aligned}
\end{equation}
where \(C_0\) is a positive constant independent of the time step \(\tau\) and \(t\).
\end{lemma}
\begin{proof}
It suffices to prove the estimate for one realization, so the index \(p\) is omitted. Let \(t_n=n\tau\),
\(\mu(\mathbf{x},t)=\chi\nabla v(\mathbf{x},t)\), and
\(\boldsymbol{\epsilon}^{(n)}=\mathbf{X}(t_n)-\widetilde{\mathbf{X}}^{(n)}\). The exact and numerical updates over
\([t_n,t_{n+1}]\) are
\begin{equation}
\begin{aligned}
\mathbf{X}(t_{n+1})&=\mathbf{X}(t_n)+\int_{t_n}^{t_{n+1}} \chi \nabla v(\mathbf{X}(t),t) \,\mathrm{d}t + \mathbf{W}(t_{n+1})-\mathbf{W}(t_n),\\
\widetilde{\mathbf{X}}^{(n+1)}&=\widetilde{\mathbf{X}}^{(n)}+\chi \tau \nabla v(\widetilde{\mathbf{X}}^{(n)},t_n) + \mathbf{W}^{(n+1)}.
\end{aligned}
\end{equation}
Since the noise is additive, the Brownian increments cancel in the error equation:
\begin{equation}
\label{eq:epsilon_start}
\boldsymbol{\epsilon}^{(n+1)}=\boldsymbol{\epsilon}^{(n)} + \int_{t_n}^{t_{n+1}} (\mu(\mathbf{X}(t),t) -\mu(\widetilde{\mathbf{X}}^{(n)},t_n))\,\mathrm{d}t.
\end{equation}
We split the integrand as
\begin{equation}
\begin{aligned}
\mu(\mathbf{X}(t),t)-\mu(\widetilde{\mathbf{X}}^{(n)},t_n)
={}&\bigl(\mu(\mathbf{X}(t),t)-\mu(\mathbf{X}(t),t_n)\bigr)
+\bigl(\mu(\mathbf{X}(t),t_n)-\mu(\mathbf{X}(t_n),t_n)\bigr)\\
&+\bigl(\mu(\mathbf{X}(t_n),t_n)-\mu(\widetilde{\mathbf{X}}^{(n)},t_n)\bigr).
\end{aligned}
\end{equation}
The first and third terms are bounded by the first-order temporal and physical smoothness of \(\mu\), giving \(J_1 \tau^2 + M_2 \chi \tau|\boldsymbol{\epsilon}^{(n)}|\).
For the middle term, Taylor expansion in space gives
\begin{equation}
\mu(\mathbf{X}(t),t_n)-\mu(\mathbf{X}(t_n),t_n)
=\mu_\mathbf{x}(\mathbf{X}(t_n),t_n)\bigl(\mathbf{X}(t)-\mathbf{X}(t_n)\bigr)
+O\bigl(|\mathbf{X}(t)-\mathbf{X}(t_n)|^2\bigr).    
\end{equation}
Using
\begin{equation}
\mathbf{X}(t)-\mathbf{X}(t_n)
=\int_{t_n}^{t}\mu(\mathbf{X}(s),s)\,\mathrm{d}s+\mathbf{W}(t)-\mathbf{W}(t_n),
\end{equation}
the drift contribution is \(O(\tau^2)\), while the Brownian part has mean-square size controlled by the standard estimates
\begin{equation}
\mathbb{E}\int_{t_n}^{t_{n+1}}\!|\mathbf{W}(t)-\mathbf{W}(t_n)|^2\,\mathrm{d}t=O(\tau^2),
\qquad
\mathbb{E}\left|\int_{t_n}^{t_{n+1}}\!(\mathbf{W}(t)-\mathbf{W}(t_n))\,\mathrm{d}t\right|^2=O(\tau^3),
\end{equation}
where the constants depend only on the diffusion coefficient and the final time. Therefore, the one-step error can be written as
\begin{equation}
\boldsymbol{\epsilon}^{(n+1)}=\boldsymbol{\epsilon}^{(n)}+R_n+S_n,\qquad
|R_n|\le C\bigl(\tau|\boldsymbol{\epsilon}^{(n)}|+\tau^2\bigr),\qquad
\mathbb{E}|S_n|^2\le C\tau^3,    
\end{equation}
with \(C\) independent of \(\tau\). Hence
\begin{equation}
\mathbb{E}|\boldsymbol{\epsilon}^{(n+1)}|^2
\le (1+C\tau)\mathbb{E}|\boldsymbol{\epsilon}^{(n)}|^2+C\tau^3 .
\end{equation}
Because \(\boldsymbol{\epsilon}^{(0)}=0\), the discrete Gronwall inequality gives \(\mathbb{E}|\boldsymbol{\epsilon}^{(n)}|^2\le C e^{Ct_n}\tau^2,\)
and therefore
\begin{equation}
\left(\mathbb{E}\bigl|\mathbf{X}(t_n)-\widetilde{\mathbf{X}}^{(n)}\bigr|^2\right)^{\frac{1}{2}}
\le e^{C_0t_n}\tau
\end{equation}
after increasing \(C_0\) if necessary. Applying the same argument to the equation for \(\log a\) proves the second estimate.
\end{proof}

Next, we analyze the discrepancy between the particles $(\widehat{\mathbf{X}}_p^{(n)},\hat{a}_p^{(n)})$ maintained by the algorithm and the auxiliary random branches $(\widetilde{\mathbf{X}}_p^{(n)},\tilde{a}_p^{(n)})$. We use the notation \(\mathbf{e}_p^{(n)}:=\widehat{\mathbf{X}}_p^{(n)}-\widetilde{\mathbf{X}}_p^{(n)}\) to denote the global error of the $p$-th particle at the $n$-th step, \(d^{(n)}:= \sum_{p=1}^{P}   \tilde{a}_p^{(n)}\left|\mathbf{e}_p^{(n)}\right|\) to denote the weighted absolute error, \(h^{(n)}:=\sum_{p=1}^{P}\left|\hat{a}_p^{(n)}-\tilde{a}_p^{(n)}\right|\) to denote the particle weight error, and \(c^{(n)}:=\\\max\left\{\int_\Omega \left|\nabla \hat{v}^{(n)}(\mathbf{x})-\nabla v^{(n)}(\mathbf{x})\right|\,\mathrm{d}\mathbf{x},\int_\Omega \left|\hat{w}^{(n)}(\mathbf{x})- w^{(n)}(\mathbf{x})\right|\,\mathrm{d}\mathbf{x}\right\}\) to denote the error of the oxygen concentration \(w\) in \(L^1\) norm and the ECM concentration \(v\) in \(W^{1,1}\)-norm. In the following lemmas, we show that the single-step growth rates of \(d^{(n)}\) and \(h^{(n)}\) are controlled by \(c^{(n)}\), while \(c^{(n)}\) itself is controlled by \(d^{(m)}\) and \(h^{(m)}\) for \(m < n\), thus allowing the use of Gronwall's inequality.
\begin{lemma}
\label{lemma:d_estimate}
The weighted absolute error \(d^{(n+1)}\) is controlled by \(d^{(n)}\) and the concentration error \(c^{(n)}\), for some \(K_2\) independent of \(L,H,P,\tau\) mentioned in Theorem \ref{thm:pic_error}:
\begin{equation}
\mathbb{E}[d^{(n+1)}] \leq (1+\tau+M_2 \chi \tau)\mathbb{E}[d^{(n)}]+e^{n\tau}K_2 \chi \tau \left(\frac{L}{H}\right)^2 + N_0\chi \tau \mathbb{E}[c^{(n)}].
\end{equation}
\end{lemma}
\begin{proof}
By comparing the update rules of \(\widehat{\mathbf{X}}^{(n)}\) \eqref{eq:u_hat_update} and \(\widetilde{\mathbf{X}}^{(n)}\) \eqref{eq:u_tilde_update}, we get:
\begin{equation}
\mathbf{e}_p^{(n+1)} = \mathbf{e}_p^{(n)} + \chi \tau \left(\nabla_H\hat{v}^{(n)}(\widehat{\mathbf{X}}_p^{(n)})-\nabla v^{(n)}(\widetilde{\mathbf{X}}_p^{(n)})\right).
\end{equation}
Use the weighted sum with the masses $a_p$ and perform the following decomposition:
\begin{equation}
\label{eq:e_update}
\begin{aligned}
d^{(n+1)}&= \sum_{p=1}^{P}   \tilde{a}_p^{(n+1)}\left|\mathbf{e}_p^{(n+1)}\right|\\
&\leq \sum_{p=1}^{P}   \tilde{a}_p^{(n+1)}\left(\left|\mathbf{e}_p^{(n)}\right|+\chi \tau\left|\nabla_H \hat{v}^{(n)}(\widehat{\mathbf{X}}_p^{(n)})-\nabla \hat{v}^{(n)}(\widehat{\mathbf{X}}_p^{(n)})\right|\right.\\
&\quad+\chi \tau\left.\left|\nabla \hat{v}^{(n)}(\widehat{\mathbf{X}}_p^{(n)})-\nabla \hat{v}^{(n)}(\widetilde{\mathbf{X}}_p^{(n)})\right|+\chi \tau\left|\nabla \hat{v}^{(n)}(\widetilde{\mathbf{X}}_p^{(n)})-\nabla v^{(n)}(\widetilde{\mathbf{X}}_p^{(n)})\right|\right).
\end{aligned}
\end{equation}
By Theorem \ref{thm:pic_error}, for some \(K_2\) independent of \(L,H,P,\tau\), we have:
\begin{equation}
\label{eq:e_update_start}
\left|\nabla_H \hat{v}^{(n)}(\widehat{\mathbf{X}}_p^{(n)})-\nabla \hat{v}^{(n)}(\widehat{\mathbf{X}}_p^{(n)})\right| \leq K_2 \left(\frac{L}{H}\right)^2.
\end{equation}
By Assumption \ref{as} which states that the second derivatives of \(\hat{v}\) are bounded by \(M_2\), we have:
\begin{equation}
\left|\nabla \hat{v}^{(n)}(\widehat{\mathbf{X}}_p^{(n)})-\nabla \hat{v}^{(n)}(\widetilde{\mathbf{X}}_p^{(n)})\right|\leq M_2 \left|\widehat{\mathbf{X}}_p^{(n)}-\widetilde{\mathbf{X}}_p^{(n)}\right|=M_2\left|\mathbf{e}_p^{(n)}\right|.
\end{equation}
For the last part \(Q_n:=\sum_{p=1}^P\tilde{a}_p^{(n+1)}\left|\nabla \hat{v}^{(n)}(\widetilde{\mathbf{X}}_p^{(n)})-\nabla v^{(n)}(\widetilde{\mathbf{X}}_p^{(n)})\right|\), we need to use the property that \(\widetilde{\mathbf{X}}_p^{(n)}\) are independently and identically distributed, and use the difference in the \(L^1\)-norms of \(\nabla \hat{v}^{(n)}\) and \(\nabla v^{(n)}\) to bound it:
\begin{equation}
\label{eq:e_update_end}
\begin{aligned}
\mathbb{E}[Q_n]&=\mathbb{E}\left[\sum_{p=1}^P\tilde{a}_p^{(n+1)}\left|\nabla \hat{v}^{(n)}(\widetilde{\mathbf{X}}_p^{(n)})-\nabla v^{(n)}(\widetilde{\mathbf{X}}_p^{(n)})\right|\right]\\
&= \mathbb{E}\left[\int_\Omega u^{(n+1)}(\mathbf{x})\left|\nabla \hat{v}^{(n)}(\mathbf{x})-\nabla v^{(n)}(\mathbf{x})\right| \,\mathrm{d}\mathbf{x}\right]\\
&\leq \left\|u^{(n+1)}\right\|_\infty\mathbb{E}\left[  \left\|\nabla \hat{v}^{(n)}-\nabla v^{(n)}\right\|_1 \right] \leq N_0\,\mathbb{E}[c^{(n)}].
\end{aligned}
\end{equation}
Substituting \eqref{eq:e_update_start}-\eqref{eq:e_update_end} into \eqref{eq:e_update} and applying Lemma \ref{lemma:mass} and Corollary \ref{corollary:mass}, we obtain:
\begin{equation}
\begin{aligned}
\mathbb{E}[d^{(n+1)}] \leq (1+\tau+M_2 \chi \tau)\mathbb{E}[d^{(n)}]+e^{n\tau}K_2 \chi \tau \left(\frac{L}{H}\right)^2 + N_0\chi \tau \mathbb{E}[c^{(n)}],
\end{aligned}
\end{equation}
which is the desired result.
\end{proof}
\begin{lemma}
The particle weight error \(h^{(n+1)}\) is controlled by \(h^{(n)}\), the weighted absolute error \(d^{(n)}\) and the concentration error \(c^{(n)}\), for some \(K_2\) independent of \(L,H,P,\tau\) mentioned in Theorem \ref{thm:pic_error}:
\begin{equation}
\label{lemma:h_estimate}
\mathbb{E}[h^{(n+1)}] \leq (1+\tau)\mathbb{E}[h^{(n)}]+2M_0\tau e^{n\tau}\left(M_1  \mathbb{E}[d^{(n)}]+K_2 \left(\frac{L}{H}\right)^2 \right)+ 2 N_0 \tau \mathbb{E}[c^{(n)}].
\end{equation}
\end{lemma}
\begin{proof}
By comparing the update rules of \(\hat{a}^{(n)}\) \eqref{eq:u_hat_update} and \(\tilde{a}^{(n)}\) \eqref{eq:u_tilde_update}, we get:
\begin{equation}
\label{eq:a_update}
h^{(n+1)}\leq h^{(n)}+\tau \sum_{p=1}^P \left|\hat{a}_p^{(n)}\left(\rho(\hat{w}_H^{(n)}(\widehat{\mathbf{X}}_p^{(n)}))-1\right)-\tilde{a}_p^{(n)}\left(\rho(w^{(n)}(\widetilde{\mathbf{X}}_p^{(n)}))-1\right)\right|.
\end{equation}
Since \(|\rho(w_1)-\rho(w_2)| \leq 2|w_1 - w_2|\) holds for \(\rho(w)=\frac{2w}{1+w}\) and any \(w_1,w_2>0\), we have:
\begin{equation}
\label{eq:a_update_start}
\begin{aligned}
&\sum_{p=1}^P \left|\tilde{a}_p^{(n)}\left(\rho(\hat{w}^{(n)}(\widehat{\mathbf{X}}_p^{(n)}))-1\right)-\tilde{a}_p^{(n)}\left(\rho(w^{(n)}(\widetilde{\mathbf{X}}_p^{(n)}))-1\right)\right| \\
\leq \, &2 \sum_{p=1}^P \left|\tilde{a}_p^{(n)} \left(\hat{w}^{(n)}(\widehat{\mathbf{X}}_p^{(n)})-w^{(n)}(\widetilde{\mathbf{X}}_p^{(n)})\right)\right| \\
\leq \, & 2\sum_{p=1}^P \left(\tilde{a}_p^{(n)} M_1\left| \widehat{\mathbf{X}}_p^{(n)}-\widetilde{\mathbf{X}}_p^{(n)}\right| + \left|\tilde{a}_p^{(n)} \left(\hat{w}^{(n)}(\widetilde{\mathbf{X}}_p^{(n)})-w^{(n)}(\widetilde{\mathbf{X}}_p^{(n)})\right)\right|\right).
\end{aligned}
\end{equation}
By Theorem \ref{thm:pic_error}, for some \(K_2\) independent of \(L,H,P,\tau\), we have:
\begin{equation}
\begin{aligned}
&\sum_{p=1}^P \left|\tilde{a}_p^{(n)}\left(\rho(\hat{w}_H^{(n)}(\widehat{\mathbf{X}}_p^{(n)}))-1\right)-\tilde{a}_p^{(n)}\left(\rho(\hat{w}^{(n)}(\widehat{\mathbf{X}}_p^{(n)}))-1\right)\right|\\ \leq \,&2\sum_{p=1}^P \tilde{a}_p^{(n)} K_2 \left(\frac{L}{H}\right)^2 \leq 2 e^{n\tau} M_0 K_2 \left(\frac{L}{H}\right)^2.
\end{aligned}
\end{equation}
Since  \(|\rho(w)-1|\leq 1\) holds for any \(w>0\), we have:
\begin{equation}
\begin{aligned}
&\sum_{p=1}^P \left|\hat{a}_p^{(n)}\left(\rho(\hat{w}_H^{(n)}(\widehat{\mathbf{X}}_p^{(n)}))-1\right)-\tilde{a}_p^{(n)}\left(\rho(\hat{w}_H^{(n)}(\widehat{\mathbf{X}}_p^{(n)}))-1\right)\right|\\ \leq \,&\sum_{p=1}^P \left|\hat{a}_p^{(n)} - \tilde{a}_p^{(n)}\right| = h^{(n)}.
\end{aligned}
\end{equation}
Using the property that \(\widetilde{\mathbf{X}}_p^{(n)}\) are distributed independently and identically, we obtain:
\begin{equation}
\label{eq:a_update_end}
\begin{aligned}
&\mathbb{E}\left[\sum_{p=1}^P \left|\tilde{a}_p^{(n)} \left(\hat{w}^{(n)}(\widetilde{\mathbf{X}}_p^{(n)})-w^{(n)}(\widetilde{\mathbf{X}}_p^{(n)})\right)\right|\right]\\
=\,&\mathbb{E}\left[\int_\Omega u^{(n)}(\mathbf{x})\left|\hat{w}^{(n)}(\mathbf{x})-w^{(n)}(\mathbf{x})\right| \,\mathrm{d}\mathbf{x}\right]\\
\leq\,& \left\|u^{(n)}\right\|_\infty \mathbb{E}\left[\left\|\hat{w}^{(n)}-w^{(n)}\right\|_1\right] \leq N_0\,\mathbb{E}[c^{(n)}].
\end{aligned}
\end{equation}
Substituting \eqref{eq:a_update_start}-\eqref{eq:a_update_end} into \eqref{eq:a_update}, we obtain:
\begin{equation}
\mathbb{E}[h^{(n+1)}] \leq (1+\tau)\mathbb{E}[h^{(n)}]+2M_0\tau e^{n\tau}\left(M_1  \mathbb{E}[d^{(n)}]+K_2 \left(\frac{L}{H}\right)^2 \right)+ 2 N_0 \tau \mathbb{E}[c^{(n)}],
\end{equation}
which is the desired result.
\end{proof}

\begin{lemma}
\label{lemma:c_estimate}
The concentration error \(c^{(n)}\) is controlled by the weighted absolute error \(d^{(n)}\) and the particle weight error \(h^{(n)}\):
\begin{equation}
\begin{aligned}
c^{(n)} \leq&  C_5e^{2n\tau}\left( H_0^\frac{9}{2} H^{-4} L+H_0^{\frac{1}{2}}P^{-\frac{1}{2}}L +H_0^{-2}L^\frac{7}{2}\right)\\
&+2 n\tau \cdot \tau \sum_{0\leq s < n} c^{(s)} +  \sum_{0\leq s < n}\left(d^{(s)}F^{(s,n)}+h^{(s)}G^{(s,n)}\right),
\end{aligned}
\end{equation}
where the nonnegative coefficients \(F^{(s,n)},G^{(s,n)}\) satisfy\\ \(\max\left\{\sum_{s=0}^{n-1} F^{(s,n)},\sum_{s=0}^{n-1} G^{(s,n)}\right\}\leq C_6(n\tau+\sqrt{n\tau})\), and \(C_5,C_6\) are constants independent of \(L,H,P,\tau\), specified in \eqref{eq:c5c6}.
\end{lemma}
\begin{proof}
Based on the properties of the system \eqref{eq:haptotaxis-system}, we will estimate the error in \(L^1\) or \(W^{1,1}\) norm for different types of concentration sequentially.

\textbf{MDE concentration \(m\).}
The update of MDE concentration \(k_m\) and the algorithm maintained \(\hat{k}_m\) in the frequency domain satisfies the following equations:
\begin{equation}
\begin{aligned}
k_{m\mathbf{q}}^{(n+1)}&= \frac{1}{1+\tau \beta + \tau D_m \left|\mathbf{q}\frac{2\pi}{L}\right|^2}\left(k_{m\mathbf{q}}^{(n)}+\tau \mathcal{F}[u^{(n)}]_\mathbf{q}\right),\\
\hat{k}_{m\mathbf{q}}^{(n+1)}&= \frac{1}{1+\tau \beta + \tau D_m \left|\mathbf{q}\frac{2\pi}{L}\right|^2}\left(\hat{k}_{m\mathbf{q}}^{(n)}+\tau \mathcal{F}_H[\hat{u}^{(n)}]_\mathbf{q}\right).
\end{aligned}
\end{equation}
The differences between $k_m^{(n+1)}$ and $\hat{k}_m^{(n+1)}$ arise from the following sources: the interpolation error inherent in the PIC method, the sample point offset error originating from the discrepancy between $\widehat{\mathbf{X}}_p^{(s)}$ and $\widetilde{\mathbf{X}}_p^{(s)}$ for \(0 \leq s \leq n\), and the random sampling error between $\tilde{u}^{(s)}$ and $u^{(s)}$. We will provide rigorous estimates for these errors.

Denote the empirical concentration \(\overline{m}\) as follows, without performing the PIC transfer (Alg.\,\ref{alg:p2g}):
\begin{equation}
\begin{aligned}
\overline{k}_{m\mathbf{q}}^{(0)}:=\hat{k}_{m\mathbf{q}}^{(0)},\quad&\overline{k}_{m\mathbf{q}}^{(n+1)}:= \frac{1}{1+\tau \beta + \tau D_m \left|\mathbf{q}\frac{2\pi}{L}\right|^2}\left(\overline{k}_{m\mathbf{q}}^{(n)}+\tau \mathcal{F}[\hat{u}^{(n)}]_\mathbf{q}\right),\\
&\overline{m}^{(n)}(\mathbf{x}) := \sum_{\mathbf{q} \in U_{H_0}} \overline{k}_{m\mathbf{q}}^{(n)} e^{i \frac{2\pi}{L} \mathbf{q} \cdot \mathbf{x}}.
\end{aligned}
\end{equation}
Taking the difference yields:
\begin{equation}
\label{eq:freq_estimate}
\begin{aligned}
\left|\hat{k}_{m\mathbf{q}}^{(n+1)}-\overline{k}_{m\mathbf{q}}^{(n+1)}\right| &\leq \tau \sum_{s=0}^n \kappa_\mathbf{q}^{n+1-s}\left|\mathcal{F}_H[\hat{u}^{(s)}]_\mathbf{q}-\mathcal{F}[\hat{u}^{(s)}]_\mathbf{q}\right|\\
&\leq \tau C_4 L^{-3}\left|\mathbf{q}\frac{2\pi}{L}\right|^4 \left(\frac{L}{H}\right)^4 \sum_{s=0}^n \kappa_\mathbf{q}^{n+1-s}\\
&\leq \frac{\left|\mathbf{q}\frac{2\pi}{L}\right|^4}{\beta+D_m\left|\mathbf{q}\frac{2\pi}{L}\right|^2} \cdot C_4 \left(\frac{L}{H}\right)^4 L^{-3},
\end{aligned} 
\end{equation}
where \(\kappa_\mathbf{q}:=\frac{1}{1+\tau \beta + \tau D_m \left|\mathbf{q}\frac{2\pi}{L}\right|^2}\) is the decay factor of frequency \(\mathbf{q}\), and the constant \(C_4\) mentioned in Theorem \ref{thm:pic_error} is independent of \(L,H,P,\tau\).

The global interpolation error \(\left\|\hat{m}^{(n+1)}-\overline{m}^{(n+1)}\right\|_1\) is bounded by the following estimate using Parseval's Theorem.
\begin{equation}
\begin{aligned}
\left\|\hat{m}^{(n+1)}-\overline{m}^{(n+1)}\right\|_1 &=\int_\Omega \left|\hat{m}^{(n+1)}(\mathbf{x})-\overline{m}^{(n+1)}(\mathbf{x})\right|\,\mathrm{d}\mathbf{x}\\
&\leq \left(\int_\Omega \left|\hat{m}^{(n+1)}(\mathbf{x})-\overline{m}^{(n+1)}(\mathbf{x})\right|^2\,\mathrm{d}\mathbf{x}\right)^\frac{1}{2} \left(\int_\Omega 1\,\mathrm{d}\mathbf{x}\right)^\frac{1}{2}\\
&= \left(L^3 \sum_{\mathbf{q}\in U_{H_0}}\left|\hat{k}_{m\mathbf{q}}^{(n+1)}-\overline{k}_{m\mathbf{q}}^{(n+1)}\right|^2\right)^\frac{1}{2} \cdot L^{\frac{3}{2}}.
\end{aligned}
\end{equation}
Substituting the estimate for each frequency in \eqref{eq:freq_estimate}, we obtain:
\begin{equation}
\begin{aligned}
\left\|\hat{m}^{(n+1)}-\overline{m}^{(n+1)}\right\|_1 &\leq \left( \sum_{\mathbf{q}\in U_{H_0}} \left(\frac{\left|\mathbf{q}\frac{2\pi}{L}\right|^4}{\beta+D_m\left|\mathbf{q}\frac{2\pi}{L}\right|^2} \cdot C_4 \left(\frac{L}{H}\right)^4\right)^2 \right)^\frac{1}{2}\\
&\leq \left(H_0^3\cdot\frac{\left(H_0\frac{2\pi}{L}\right)^4}{D_m^2}\cdot C_4^2 \left(\frac{L}{H}\right)^8\right)^\frac{1}{2} \\
&\leq C_1 H_0^\frac{7}{2} H^{-4} L^2,
\end{aligned}
\end{equation}
where \(C_1=8\pi^3\frac{C_4}{D_m}\) is independent of \(L,H,P,\tau\).

The global interpolation error of the concentration gradient can be controlled by the same method, with an additional factor of the maximum frequency \(H_0 \cdot \frac{2\pi}{L}\):
\begin{equation}
\begin{aligned}
\left\|\nabla \hat{m}^{(n+1)}-\nabla \overline{m}^{(n+1)}\right\|_1 \leq C_1 H_0^\frac{9}{2} H^{-4} L.
\end{aligned}
\end{equation}
The sample point offset error originating from the discrepancy between $\widehat{\mathbf{X}}_p^{(s)}$ and $\widetilde{\mathbf{X}}_p^{(s)}$ is bounded by the following estimate. Notice that
\begin{equation}
\hat{m}^{(n+1)}(\mathbf{x}) - \tilde{m}^{(n+1)}(\mathbf{x}) = \sum_{0 \leq s \leq n} \tau e^{- \beta(n+1-s) \tau}\bigl( (\hat{u}^{(s)} - \tilde{u}^{(s)}) \ast K_{(n+1-s)\tau D_m+\sigma} \bigr)(\mathbf{x}),
\end{equation}
where \(K_\epsilon\) is the convolution kernel corresponding to the three-dimensional normal distribution \(N(0,\epsilon I_3)\).

For two distinct points \(\mathbf{X}_1,\mathbf{X}_2 \in \Omega\), denote the Dirac functions \(\delta_1(\mathbf{x})=\delta(\mathbf{x}-\mathbf{X}_1),\delta_2(\mathbf{x})=\delta(\mathbf{x}-\mathbf{X}_2)\). The following conclusion can be obtained through calculation, where \(R:=|\mathbf{X}_1-\mathbf{X}_2|\) is the distance:
\begin{equation}
\begin{aligned}
\left\|\delta_1 \ast K_\epsilon\right\|_1=1,\quad\left\|(\delta_1-\delta_2) \ast K_\epsilon\right\|_1 \leq 4R \max \{\epsilon^{-\frac{1}{2}},1\},\\
\left\|(\delta_1-\delta_2) \ast \nabla K_\epsilon\right\|_1 \leq 4R \max \{\epsilon^{-1},1\}.
\end{aligned}
\end{equation}
The discrepancy between the empirical distribution convolved with a Gaussian kernel can be controlled by the pairwise distances between corresponding points:
\begin{equation}
\begin{aligned}
\left\|(\hat{u}^{(s)}-\tilde{u}^{(s)})\ast K_\epsilon \right\|_1 &= \left\|\sum_{p=1}^P\left( \hat{a}_p^{(s)} \delta\left(\mathbf{x} - \widehat{\mathbf{X}}_p^{(s)}\right)-\tilde{a}_p^{(s)} \delta\left(\mathbf{x} - \widetilde{\mathbf{X}}_p^{(s)}\right)\right)\ast K_\epsilon \right\|_1 \\
&\leq \sum_{p=1}^P \left(\left|\hat{a}_p^{(s)}-\tilde{a}_p^{(s)}\right| + 4\tilde{a}_p^{(s)}\left|\widehat{\mathbf{X}}_p^{(s)}-\widetilde{\mathbf{X}}_p^{(s)}\right| \cdot \max \{\epsilon^{-\frac{1}{2}},1\} \right)\\
&=h^{(s)}+4 d^{(s)}\cdot \max \{\epsilon^{-\frac{1}{2}},1\},\\
\left\|(\hat{u}^{(s)}-\tilde{u}^{(s)})\ast \nabla K_\epsilon \right\|_1 & \leq h^{(s)} \cdot \epsilon^{-\frac{1}{2}}+4 d^{(s)}\cdot \max \{\epsilon^{-1},1\}.
\end{aligned}
\end{equation}
Accumulating over time, we obtain:
\begin{equation}
\begin{aligned}
&\|\hat{m}^{(n+1)} - \tilde{m}^{(n+1)}\|_1 \leq \sum_{0 \leq s \leq n} \tau \left\| (\hat{u}^{(n+1-s)} - u^{(n+1-s)}) \ast K_{s\tau D_m+\sigma} \right\|_1\\
&\leq \tau \sum_{0 \leq s \leq n} \left(h^{(n+1-s)} +4d^{(n+1-s)} \max\left\{\left(s\tau D_m+\sigma\right)^{-\frac{1}{2}},1\right\}\right),\\
&\|\nabla\hat{m}^{(n+1)} - \nabla \tilde{m}^{(n+1)}\|_1 \leq \sum_{0 \leq s \leq n} \tau \left\| (\hat{u}^{(n+1-s)} - u^{(n+1-s)}) \ast \nabla K_{s\tau D_m+\sigma} \right\|_1\\
&\leq \tau \sum_{0 \leq s \leq n} \left(h^{(n+1-s)}\cdot \left(s\tau D_m+\sigma\right)^{-\frac{1}{2}} +4d^{(n+1-s)} \max\left\{\left(s\tau D_m+\sigma\right)^{-1},1\right\}\right).
\end{aligned}
\end{equation}

The random sampling error consists of two parts: truncation in the high-frequency component and variance in the low-frequency component. Specifically, we have the following estimate (see\cite{hu2025fast} for complete details):
\begin{equation}
\begin{aligned}
\left\|\tilde{m}^{(n+1)}-  m^{(n+1)}\right\|_1 & \leq L^{\frac{3}{2}}\left(\left\|\tilde{m}^{(n+1)}- m^{(n+1)}\right\|_2\right)^\frac{1}{2}\\
&= \biggl( L^6 \sum_{\mathbf{q} \in \mathbb{Z}^3 \setminus U_{H_0}}\left| k^{(n)}_{m\mathbf{q}} \right|^2  + \sum_{\mathbf{q} \in U_{H_0}} \left| \tilde{k}^{(n)}_{m\mathbf{q}} - k^{(n)}_{m\mathbf{q}} \right|^2 \biggr)^\frac{1}{2}\\
&\leq \left(\left(H_0 \frac{2\pi}{L}\right)^{-6} L^3 N_{1,2}+\frac{(M_0 e^{n\tau})^2}{PH_0}L^4\right)^\frac{1}{2},\\
\left\|\nabla\tilde{m}^{(n+1)}- \nabla m^{(n+1)}\right\|_1 
&\leq \left(\left(H_0 \frac{2\pi}{L}\right)^{-4} L^3 N_{1,2}+\frac{(M_0 e^{n\tau})^2 H_0}{P}L^2\right)^\frac{1}{2}.
\end{aligned}
\end{equation}
Summing up all sources of error, we obtain:
\begin{equation}
\begin{aligned}
\left\|\hat{m}^{(n+1)}-  m^{(n+1)}\right\|_1 & \leq C_1 H_0^\frac{7}{2} H^{-4} L^2+\sum_{0\leq s \leq n}\left(d^{(s)}F_0^{(s,n+1)}+h^{(s)}G_0^{(s,n+1)}\right)\\
&\quad+H_0^{-3}L^\frac{9}{2}N_{1,2}^\frac{1}{2}+H_0^{-\frac{1}{2}}P^{-\frac{1}{2}}L^2 e^{n\tau},\\
\left\|\nabla\hat{m}^{(n+1)}- \nabla m^{(n+1)}\right\|_1 & \leq C_1 H_0^\frac{9}{2} H^{-4} L+\sum_{0\leq s \leq n}\left(d^{(s)}F_1^{(s,n+1)}+h^{(s)}G_1^{(s,n+1)}\right)\\
&\quad+H_0^{-2}L^\frac{7}{2}N_{1,2}^\frac{1}{2}+H_0^{\frac{1}{2}}P^{-\frac{1}{2}}L e^{n\tau},
\end{aligned}
\end{equation}
where the nonnegative coefficients \(F_0^{(s,n)}=\tau,G_0^{(s,n)}=4\tau ((n-s)D_m+\sigma)^{-\frac{1}{2}},F_1^{(s,n)}=\tau ((n-s)D_m+\sigma)^{-\frac{1}{2}},G_1^{(s,n)}=4\tau ((n-s)D_m+\sigma)^{-1}\) satisfy\\ \(\max\left\{\sum_{s=0}^{n-1} F_0^{(s,n)},\sum_{s=0}^{n-1} F_1^{(s,n)},\sum_{s=0}^{n-1} G_0^{(s,n)},\sum_{s=0}^{n-1} G_1^{(s,n)}\right\}\leq C_2(n\tau+\sqrt{n\tau})\) \\for \(C_2=(4+4D_m^{-1})|\log \sigma|\) independent of \(L,H,P,\tau\).

\textbf{ECM concentration \(v\).} The governing equation for \(v\) is given by \((\log v)_t = - \alpha m\), which implies that \(v\) decreases monotonically at all grid points and cannot exceed its initial value \(v_0:=\max_{\mathbf{x}\in \Omega}v(\mathbf{x},0)\) at any time. Consequently, we have the following estimate:
\begin{equation}
\begin{aligned}
\left\|\nabla\hat{v}^{(n)}- \nabla v^{(n)}\right\|_1 &= \int_\Omega \left|\hat{v}^{(n)}(\mathbf{x}) \nabla \log \hat{v}^{(n)}(\mathbf{x}) -v^{(n)}(\mathbf{x}) \nabla \log v^{(n)}(\mathbf{x})\,\mathrm{d}\mathbf{x}\right|\\
&\leq v_0 \left\|\nabla \log\hat{v}^{(n)}- \nabla \log v^{(n)}\right\|_1 + n\tau M_1 \left\|\log\hat{v}^{(n)}- \log v^{(n)}\right\|_1.
\end{aligned}
\end{equation}
From the update rule of \(v\), we have \(\log v^{(n)}=-\alpha \tau \sum_{s<n}m^{(s)}\). Hence, there exists a constant \(C_3=v_0 \alpha C_2 + \alpha M_1\) independent of \(L,H,P,\tau\), such that:
\begin{equation}
\label{eq:w11_v}
\begin{aligned}
\left\|\nabla\hat{v}^{(n)}- \nabla v^{(n)}\right\|_1 \leq& n \tau C_1 H_0^\frac{9}{2} H^{-4} L+\sum_{0\leq s < n}\left(d^{(s)}F_2^{(s,n)}+h^{(s)}G_2^{(s,n)}\right)\\
&\quad+n \tau H_0^{-2}L^\frac{7}{2}N_{1,2}^\frac{1}{2}+n \tau H_0^{\frac{1}{2}}P^{-\frac{1}{2}}L e^{n\tau},
\end{aligned}
\end{equation}
where the sum of coefficients: \(\max\left\{\sum_{s=0}^{n-1} F_2^{(s,n)},\sum_{s=0}^{n-1} G_2^{(s,n)}\right\}\leq C_3(n\tau+n^2\tau^2)\).

\textbf{Oxygen concentration w}. The governing equation for \(w\) is given by \(w_t = D_w \Delta w - \eta(u) w + \gamma v - w\). Similarly to the MDE concentration \(m\), we only need to estimate the $L^1$-norm \(\left\|\left((\gamma \hat{v}^{(s)}-\eta(\hat{u}^{(s)}) \hat{w}^{(s)})-(\gamma v^{(s)}-\eta(u^{(s)}) w^{(s)})\right)\ast K_\epsilon\right\|_1\), where \(\gamma \hat{v}^{(s)}-\eta(\hat{u}^{(s)})\) is the source term in Alg. \ref{alg:sipf-pic}. This can be decomposed as:
\begin{equation}
\begin{aligned}
\left\|\left(\hat{v}^{(s)}- v^{(s)}\right)\ast K_\epsilon\right\|_1 &\leq \left\|\hat{v}^{(s)}- v^{(s)}\right\|_1 \leq v_0 \left\|\log \hat{v}^{(s)}- \log v^{(s)}\right\|_1,\\
\left\|\left(\eta(\hat{u}^{(s)}) \hat{w}^{(s)}-\eta(u^{(s)}) w^{(s)}\right)\ast K_\epsilon\right\|_1 &\leq 2\left\|\hat{w}^{(s)}- w^{(s)}\right\|_1 + 2w_0 \left\|\left(\hat{u}^{(s)}-u^{(s)}\right) \ast K_\epsilon\right\|_1.
\end{aligned}
\end{equation}
Summing over \(s<n\) yields:
\begin{equation}
\label{eq:w11_w}
\begin{aligned}
\left\|\hat{w}^{(n)}- w^{(n)}\right\|_1 \leq& n \tau C_1 H_0^\frac{9}{2} H^{-4} L+\sum_{0\leq s < n}\left(d^{(s)}F_2^{(s,n)}+h^{(s)}G_2^{(s,n)}\right)\\
+2 n\tau^2 \sum_{0\leq s < n}& \left\|\hat{w}^{(s)}- w^{(s)}\right\|_1 +n \tau H_0^{-2}L^\frac{7}{2}N_{1,2}^\frac{1}{2}+n \tau H_0^{\frac{1}{2}}P^{-\frac{1}{2}}L e^{n\tau}.
\end{aligned}
\end{equation}
Combining the above results \eqref{eq:w11_v}, \eqref{eq:w11_w} yields the desired conclusion, where the coefficients
\begin{equation}
\label{eq:c5c6}
C_5=1+C_1+N_{1,2}^\frac{1}{2},C_6=\max\{C_2,C_3\}.
\end{equation}.
\end{proof}
The main convergence result is presented as below.
\begin{theorem}
\label{thm:overall}
For the SIPF-W method with fourth-order particle-to-grid and second-order grid-to-particle interpolation and final time T, the weighted absolute error defined as $\hat{d}^{(n)}:=\hat{a}_p\sum_{1\leq p \leq P}\left|\mathbf{X}_p(n \tau)-\widehat{\mathbf{X}}_p^{(n)}\right|$ and the particle weight error defined as $\hat{h}^{(n)}:=\sum_{1\leq p \leq P}\left|a_p(n \tau)-\hat{a}_p\right|$, satisfy the following growth bound:
\begin{equation}
\begin{aligned}
\hat{d}^{\left(N_T\right)} + \hat{h}^{\left(N_T\right)} \leq C_7e^{C_8 T} \left(\tau + H^{-4} H_0^{\frac{9}{2}}L+  P^{-\frac{1}{2}} H_0^{\frac{1}{2}}L+H_0^{-2}L^\frac{7}{2}\right),
\end{aligned}
\end{equation}
where \(C_7,C_8\) are constants independent of \(L,H,P,\tau\), specified in \eqref{eq:thm_gronwall}.
\label{eq:main_results}
\end{theorem}
\begin{proof}
Combining Lemma \ref{lemma:d_estimate}, Lemma \ref{lemma:h_estimate}, and Lemma \ref{lemma:c_estimate}, we demonstrate that the growth rates of \(d^{\left(N_T\right)}\) and \(h^{\left(N_T\right)}\) can be controlled by a Gronwall-type estimate. Specifically, let \(d^{\left(N_T\right)}\) and \(h^{\left(N_T\right)}\) satisfy the following relation, as demonstrated in the lemmas, where \(K_0,K_1\) are constants independent of \(L,H,P,\tau\):
\begin{equation}
\begin{aligned}
d^{(n+1)} & \leq (1+K_1\tau) d^{(n)} + K_0 \tau ( L^2 H_0^{-2}+c^{(n)}),\\
h^{(n+1)} & \leq (1+K_1\tau) h^{(n)} + K_0 \tau ( L^2 H_0^{-2}+c^{(n)}+d^{(n)}),\\
c^{(n+1)} & \leq 2n \tau^2 \sum_{s=0}^n c^{(s)} + K_0 (d^{(n)}+h^{(n)} + H^{-4} H_0^{\frac{9}{2}}L+  P^{-\frac{1}{2}} H_0^{\frac{1}{2}}L+H_0^{-2}L^\frac{7}{2}).
\end{aligned}
\end{equation}
Let the accumulation \(S^{(n)}:=\tau\sum_{s=0}^n c^{(s)}\) and \(\xi^{(n)}:=S^{(n)}+d^{(n)}+h^{(n)}\), then \(S^{(n+1)} \leq (1+2T\tau)S^{(n)}+\tau K_0 (d^{(n)}+h^{(n)}+H^{-4} H_0^{\frac{9}{2}}+ P^{-\frac{1}{2}} H_0^{\frac{1}{2}}+H_0^{-2})\). We have the one-step estimate:
\begin{equation}
\xi^{(n+1)} \leq \xi^{(n)}+(2T +K_1+2K_0) \tau\xi^{(n)} + 2\tau K_0( H^{-4} H_0^{\frac{9}{2}}L+  P^{-\frac{1}{2}} H_0^{\frac{1}{2}}L+H_0^{-2}L^\frac{7}{2}),
\end{equation}
which is suitable for the discrete Gronwall-type estimate, yielding the following inequality:
\begin{equation}
\label{eq:thm_gronwall}
\begin{aligned}
d^{\left(N_T\right)} + h^{\left(N_T\right)} \leq C_7e^{C_8 T} \left(H^{-4} H_0^{\frac{9}{2}}L+  P^{-\frac{1}{2}} H_0^{\frac{1}{2}}L+H_0^{-2}L^\frac{7}{2}\right).
\end{aligned}
\end{equation}

Finally, from the estimate
\begin{equation}
\label{eq:convergence_estimate}
\begin{aligned}
\hat{d}^{(N_T)} &\leq d^{(N_T)} + \sum_{p=1}^{P} \tilde{a}_p^{(N_T)} \bigl| \mathbf{X}_p(T) - \widetilde{\mathbf{X}}_p^{(N_T)} \bigr|
\leq d^{(N_T)} + \frac{e^{T}}{P}
\Biggl( \sum_{p=1}^{P} \bigl| \mathbf{X}_p(T) - \widetilde{\mathbf{X}}_p^{(N_T)} \bigr|^2 \Biggr)^{\!\frac{1}{2}},\\
\hat{h}^{(N_T)} & \leq h^{(N_T)} + \sum_{p=1}^{P} \bigl|a_p(T) - \tilde{a}_p^{(N_T)} \bigr|\leq h^{(N_T)} + \frac{e^{T}}{P}
\Biggl( \sum_{p=1}^{P} \bigl| \log a_p(T) - \log\tilde{a}_p^{(N_T)} \bigr|^2 \Biggr)^{\!\frac{1}{2}},
\end{aligned}
\end{equation}
which follows from the triangle inequality the Cauchy-Schwarz inequality and Lemma \ref{lemma:mass}, 
combining with the temporal error bound from Lemma~\ref{lemma:temporal}, we obtain the claimed convergence bound in  \eqref{eq:main_results}.
\end{proof}

\begin{corollary}
\label{cor:optimized_error}
By optimizing the frequency cutoff threshold \(H_0 = H^{\frac{8}{13}}\) and the particle count \(P\geq H^{\frac{40}{13}}\), the error bound simplifies to
\begin{equation}
\begin{aligned}
\hat{d}^{\left(N_T\right)}+\hat{h}^{\left(N_T\right)} &\leq C_7e^{C_8 T} \left(\tau +  H^{-\frac{16}{13}}\left(L+L^{\frac{7}{2}}\right)\right).
\label{eq:main_results_optim}
\end{aligned}
\end{equation}
\end{corollary}

\section{Numerical Experiments}
\label{sec:Num}
In this section, we present a 2D numerical example (Section \ref{subsec:2D}) to compare the results of our algorithm with those of an established mesh-based algorithm, thereby validating the accuracy of our method. We then extend the example to a 3D case (Section \ref{subsec:3D}) that has not been addressed in earlier work, and validate the theoretical convergence rate described in Theorem \ref{thm:overall}.

\subsection{Two-Dimensional Simulation}
\label{subsec:2D}

We begin by presenting numerical results for a two-dimensional example ($d=2$), following the benchmark example described in\cite{walker2007global}. To assess the precision of our numerical method, we first conduct a convergence study to verify  that the algorithm converges as the resolution $H$ approaches infinity, and then perform a comparison with the second-order finite-volume scheme described in\cite{chertock2008second}.
In the benchmark example, the initial conditions of the system \eqref{eq:haptotaxis-system} are specified as
\begin{equation}
\begin{aligned}
&u(\mathbf{x},0)=5 \max (r_0^2-(\mathbf{x} - \mathbf{x}_0)^2, 0), \quad m(\mathbf{x},0)=u(\mathbf{x},0)\\&v(\mathbf{x},0) = 0.05 \cos \left(\frac{5\pi x_1^2}{18}\right) \sin \left(\frac{13\pi x_2^2}{72}\right) + 0.3, \quad w(\mathbf{x},0) = 4v(\mathbf{x},0),
\end{aligned}
\end{equation}
where the center of mass \(\mathbf{x}_0=(3,3)\), initial radius \(r_0 = \sqrt{0.3}\), physical domain $\Omega = [0, 6]^2$. The parameters are selected as \(\chi = 0.4,D_u = D_m = D_w = 0.01, \alpha = 5,\beta = 0.01,\gamma = 5\). The final time is set as $T = 6.0$, and the timestep $\tau = 10^{-3}$ (total steps $N_T=6000$). The particle count for our algorithm is set as \(P = 2^{22}\). When plotting the density map, the weights carried by the particles are mapped onto a $128\times 128$ spatial grid using the Particle-to-Grid transfer (Alg. \ref{alg:p2g}).

The evolution of the cell population $u(\mathbf{x},t)$ calculated by Alg. \ref{alg:sipf-pic} is visualized in Figure~\ref{fig:u_evolution}, which shows the concentration of $u$ at six equally spaced time points from $t=0$ to $t=T$. The results illustrate the rapid spread of cells driven by the haptotactic flux and the nonlinear reaction term, consistent with the dynamics reported in\cite{walker2007global}.

\begin{figure}[!ht]
    \centering
    
    \begin{subfigure}[b]{0.32\textwidth}
        \centering
        \includegraphics[width=\textwidth]{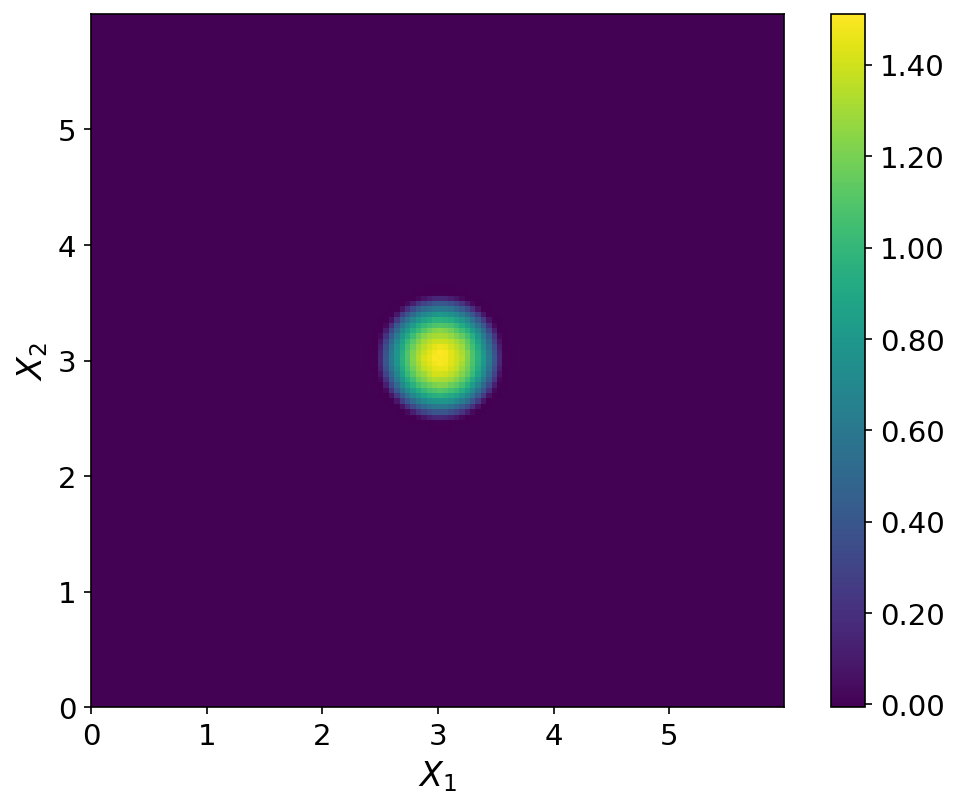}
        \caption{$t = 0.0$}
        \label{fig:u_t00}
    \end{subfigure}
    \hfill
    \begin{subfigure}[b]{0.32\textwidth}
        \centering
        \includegraphics[width=\textwidth]{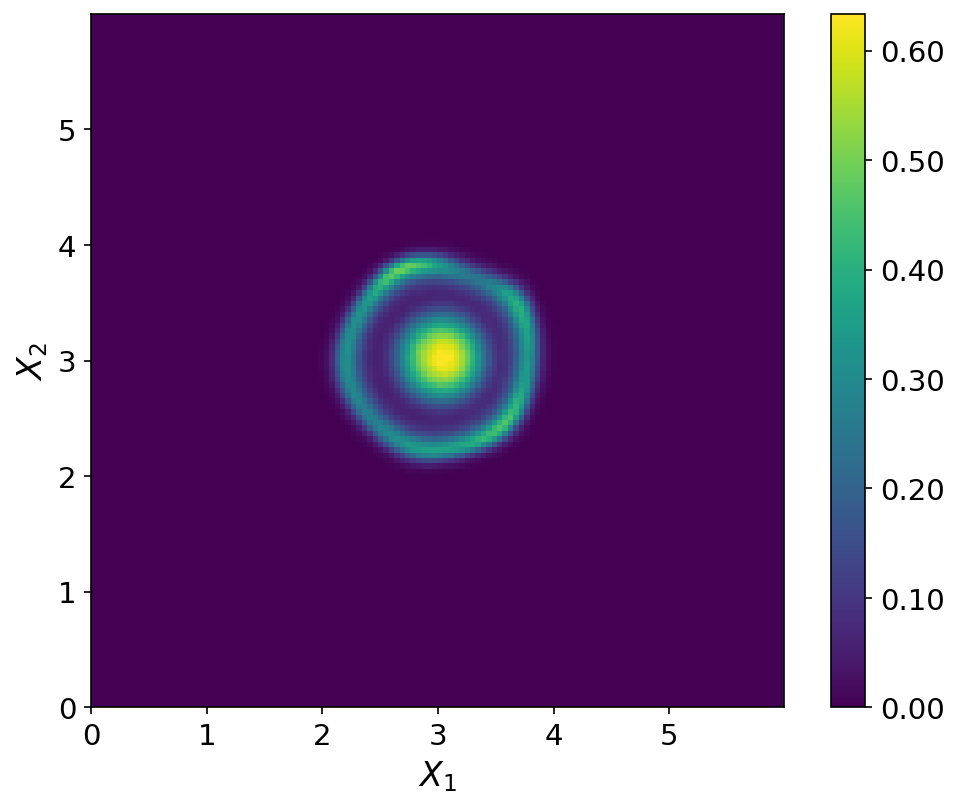}
        \caption{$t = 1.2$}
        \label{fig:u_t12}
    \end{subfigure}
    \hfill
    \begin{subfigure}[b]{0.32\textwidth}
        \centering
        \includegraphics[width=\textwidth]{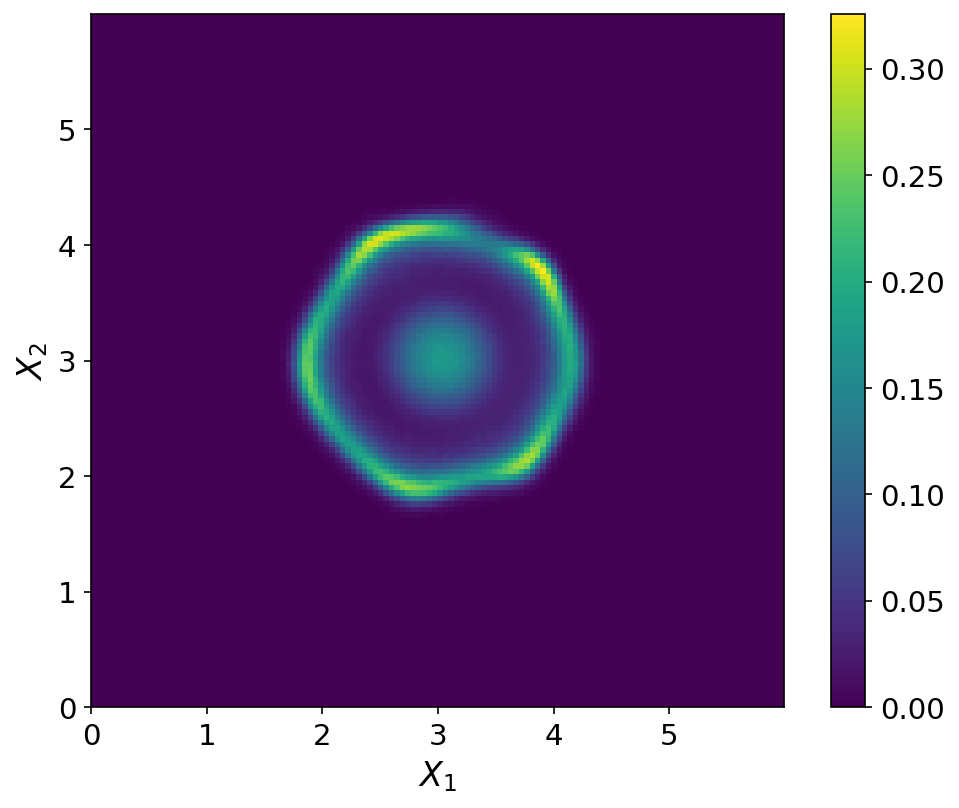}
        \caption{$t = 2.4$}
        \label{fig:u_t24}
    \end{subfigure}
    
    \vspace{0.5cm}
    
    \begin{subfigure}[b]{0.32\textwidth}
        \centering
        \includegraphics[width=\textwidth]{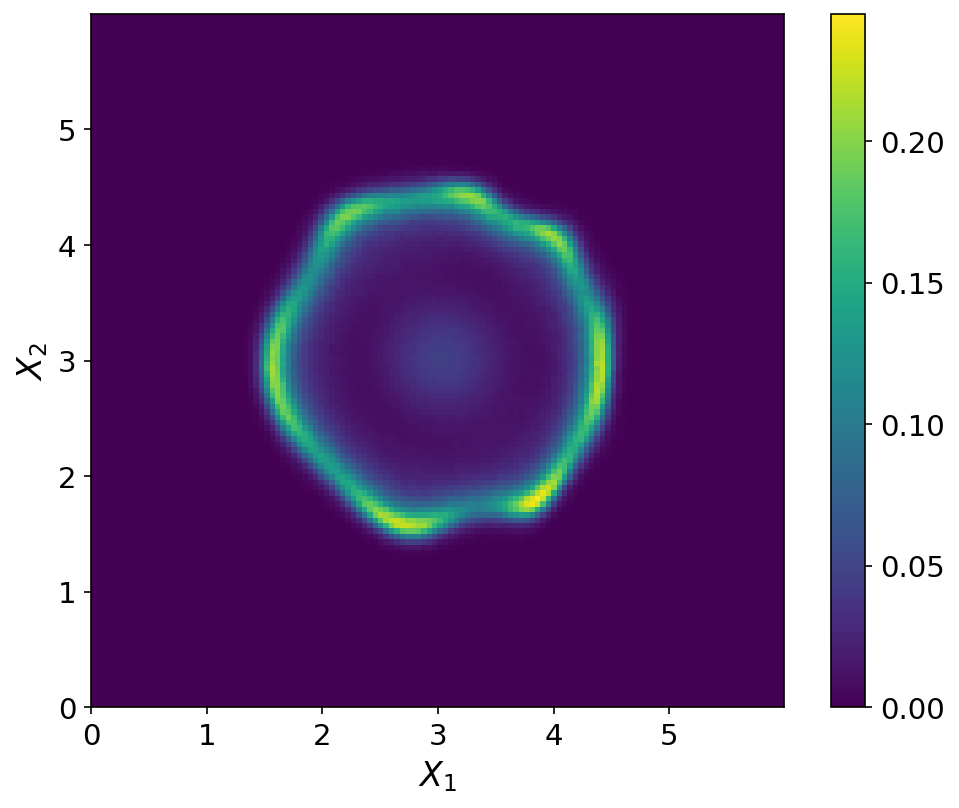}
        \caption{$t = 3.6$}
        \label{fig:u_t36}
    \end{subfigure}
    \hfill
    \begin{subfigure}[b]{0.32\textwidth}
        \centering
        \includegraphics[width=\textwidth]{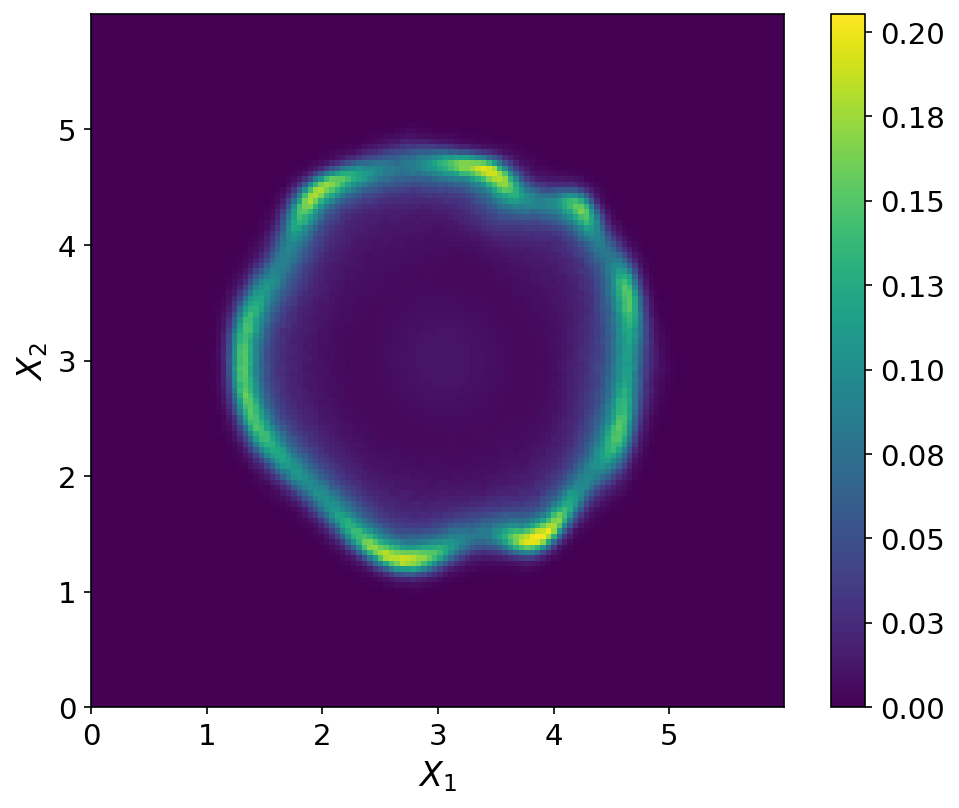}
        \caption{$t = 4.8$}
        \label{fig:u_t48}
    \end{subfigure}
    \hfill
    \begin{subfigure}[b]{0.32\textwidth}
        \centering
        \includegraphics[width=\textwidth]{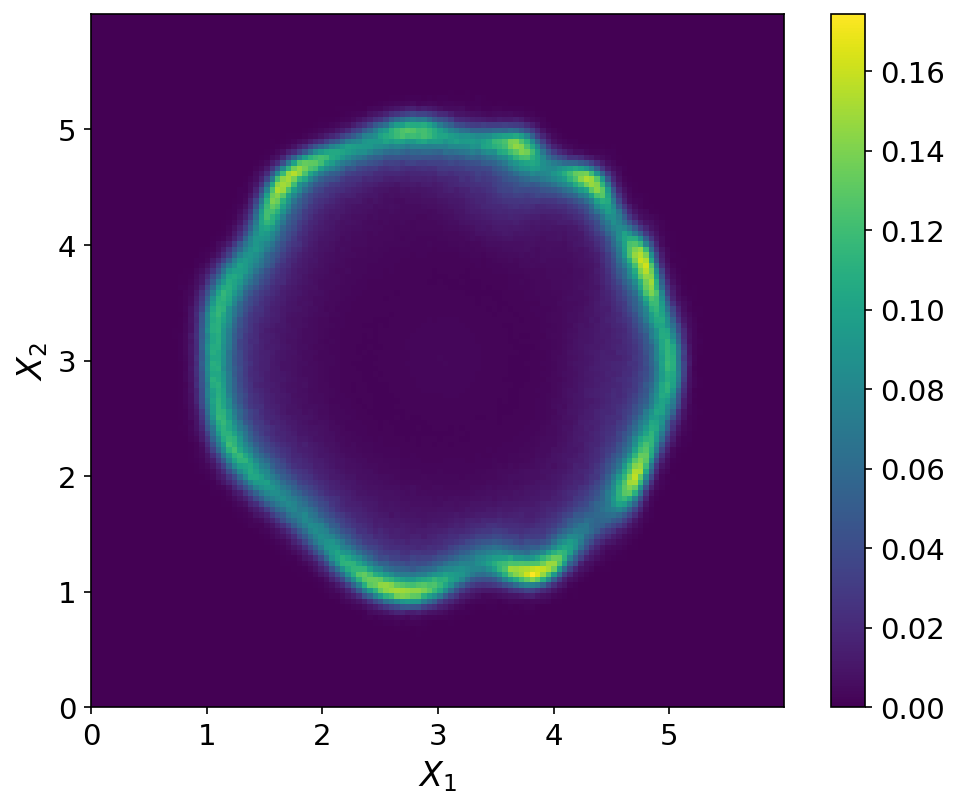}
        \caption{$t = 6.0$}
        \label{fig:u_t60}
    \end{subfigure}
    
    \caption{Evolution of the cell density $u$.}
    \label{fig:u_evolution}
\end{figure}

Convergence studies for the SIPF-W method are conducted by comparing the highest resolution of $H=2^{10}$ with other resolutions $H=2^6,...,2^9$. In these simulations, we set $H_0 = H$, meaning that no frequency cutoff is applied. To allow a fair comparison across different grid sizes, all solution fields are downsampled via Fourier truncation to a $16 \times 16$ spatial grid, and the relative difference of \(u\) is quantified using the $L^2$ norm at the final time \(T\), where the length of the physical domain is \(L=6.0\):
\begin{equation}
\mathcal{E} := \frac{\|u_1-u_2\|_2}{\max (\|u_1\|_2,\|u_2\|_2)},\quad\|u\|_2:=\frac{L}{16}\left(\sum_{0\leq i,j \leq 15}u_{ij}^2\right)^{\frac{1}{2}}.
\end{equation}

The numerical convergence of our method as the resolution increases from $H=2^6$ to $H=2^{10}$ is shown in Figure~\ref{fig:alg2alg}. The spatial convergence order is validated through the slope of $\log_2 \mathcal{E}$ versus $\log_2 H$, with a measured value of $-2.109$, which is higher than the predicted theoretical order of $16/13$ (Theorem \ref{thm:overall}).

\begin{figure}[htbp]
    \centering

    \begin{subfigure}[t]{0.45\textwidth}
        \centering
        \includegraphics[width=\textwidth]{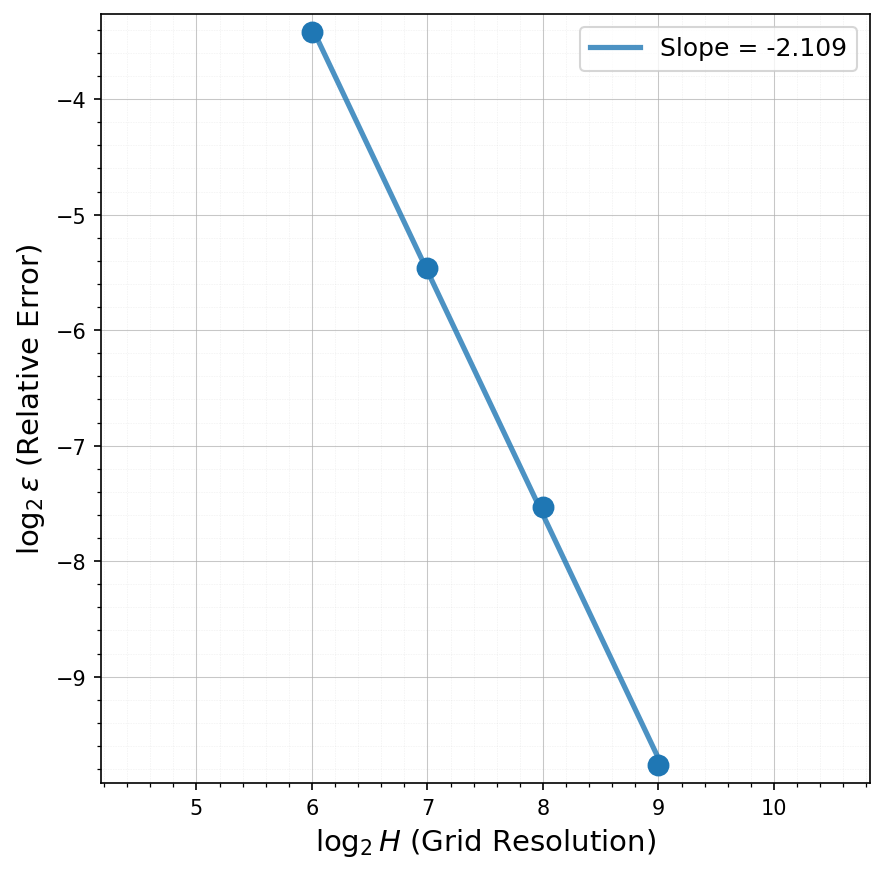}
        \caption{Numerical convergence of our method.}
        \label{fig:alg2alg}
    \end{subfigure}
    \hfill
    \begin{subfigure}[t]{0.45\textwidth}
        \centering
        \includegraphics[width=\textwidth]{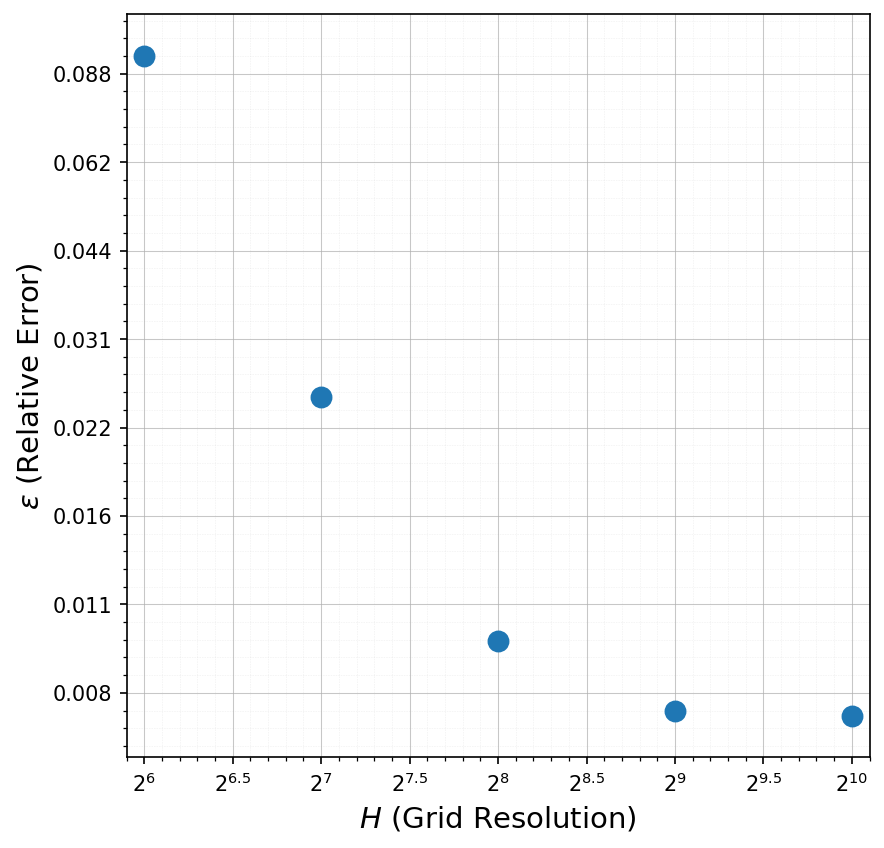}
        \caption{Comparison between mesh-based algorithm and our method.}
        \label{fig:mesh2alg}
    \end{subfigure}

    \caption{Numerical convergence of our method, and comparison with a second-order finite-volume scheme.}
    \label{fig:resolution_study}
\end{figure}

In addition, we compare the numerical results of the second-order finite-volume scheme described in\cite{chertock2008second} (with the maximum resolution \(H_{\text{finite-volume}}=2^9\) that remains numerically stable; finer grids trigger instability due to CFL condition violation) to the results of our algorithm at resolutions of \(H=2^6\) to \(H=2^{10}\). As shown in Figure~\ref{fig:mesh2alg}, the relative difference between our algorithm and the second-order finite-volume scheme is within 1\%. This close agreement demonstrates the high accuracy of our method relative to a well-established reference scheme.

\subsection{Three-Dimensional Simulation}
\label{subsec:3D}
We extend the benchmark example to a three-dimensional case ($d=3$). Unlike previous studies that focused primarily on 2D configurations, the numerical computation for the 3D problem has not been addressed in earlier work.

For the periodic domain \(\Omega = [0,6]^3\), the initial conditions are specified as
\begin{equation}
\begin{aligned}
&u(\mathbf{x},0)=3 \cdot \mathbf{1}_{|\mathbf{x} - \mathbf{x}_0|\leq r_0}, \quad m(\mathbf{x},0)=u(\mathbf{x},0),\\&v(\mathbf{x},0) = 0.05 \cos \left(\frac{5\pi x_1^2}{18}\right) \sin \left(\frac{13\pi x_2^2}{72}\right)\cos \left(\frac{\pi x_3^2}{12}\right) + 0.3, \quad w(\mathbf{x},0) = 4v(\mathbf{x},0),
\end{aligned}
\end{equation}
where the center of mass \(\mathbf{x}_0=(3,3,3)\), initial radius \(r_0 = \sqrt{0.3}\), physical domain $\Omega = [0, L]^3\quad(L=6.0)$. The parameter selection is the same as in the 2D case: \(\chi = 0.4,D_u = D_m = D_w = 0.01, \alpha = 5,\beta = 0.01,\gamma = 5\). The final time is set as $T = 6.0$, and the timestep $\tau = 10^{-3}$ (total steps $N_T=6000$). The particle count for our algorithm is set as \(P = 2^{22}\). When plotting the density map, the weights carried by the particles are mapped onto a $32\times 32 \times 32$ spatial grid using the Particle-to-Grid transfer (Alg. \ref{alg:p2g}).

In the 3D numerical example, a rapid spread of cells driven by haptotactic flux emerges, similar to the behavior observed in the 2D case. Figure~\ref{fig:3d_1} shows the evolution of the cell population $u(\mathbf{x},t)$ on the cross-section at $x_3=\frac{L}{2}$. Figure~\ref{fig:3d_2} shows 3D plots of the evolution of $u(\mathbf{x},t)$ on three mutually perpendicular cross-sections at $x_1=\frac{L}{2}$, $x_2=\frac{L}{2}$, and $x_3=\frac{L}{2}$.

\begin{figure}[!ht]
    \centering
    
    \begin{subfigure}[b]{0.32\textwidth}
        \centering
        \includegraphics[width=\textwidth]{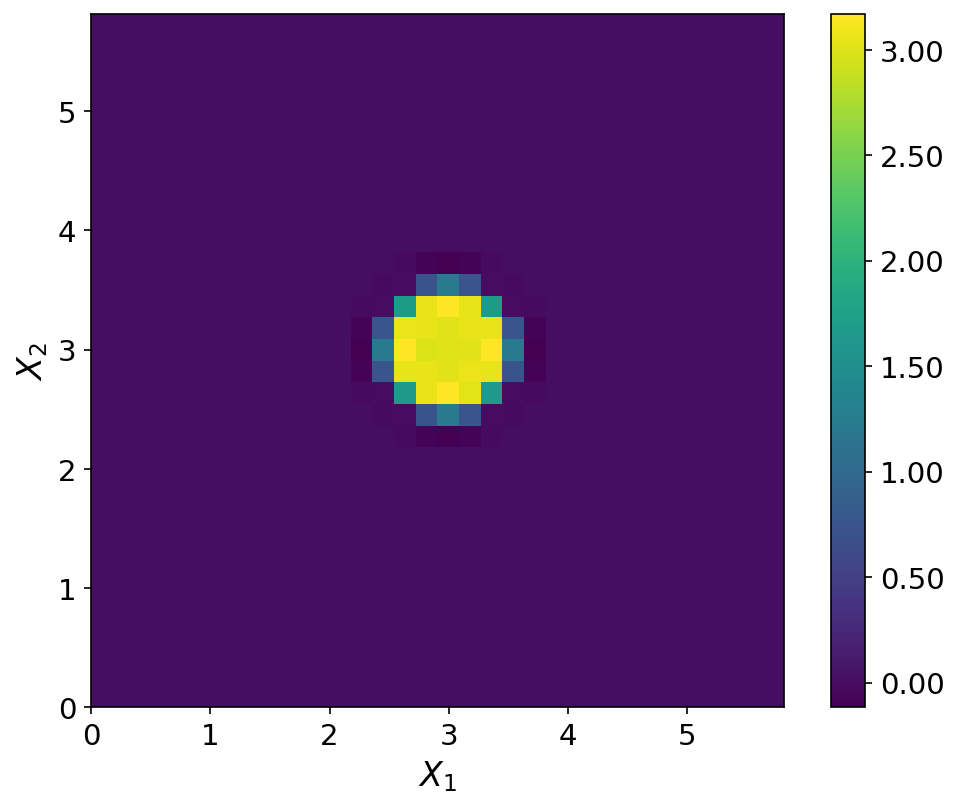}
        \caption{$t = 0.0$}
        \label{fig:u3_t00}
    \end{subfigure}
    \hfill
    \begin{subfigure}[b]{0.32\textwidth}
        \centering
        \includegraphics[width=\textwidth]{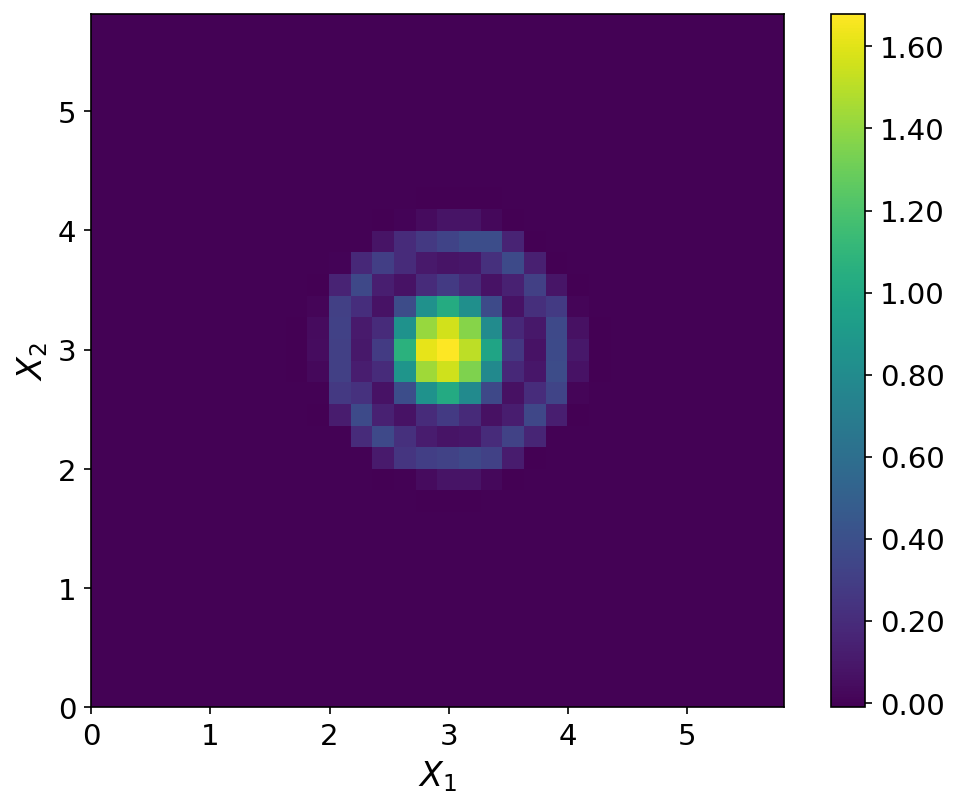}
        \caption{$t = 1.2$}
        \label{fig:u3_t12}
    \end{subfigure}
    \hfill
    \begin{subfigure}[b]{0.32\textwidth}
        \centering
        \includegraphics[width=\textwidth]{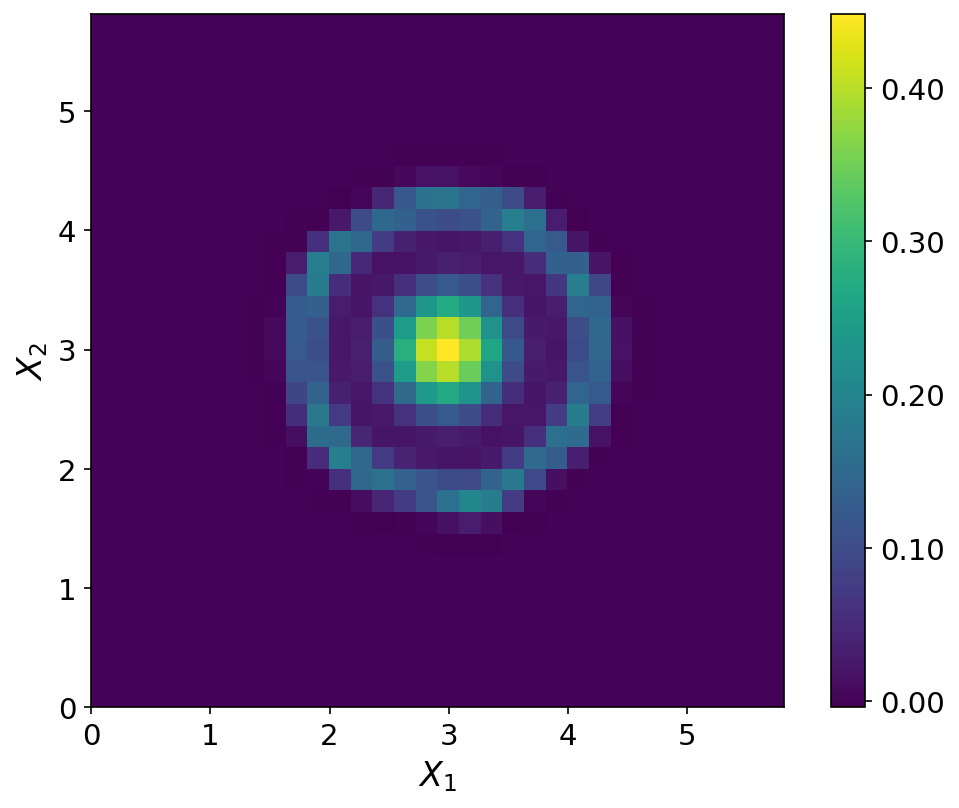}
        \caption{$t = 2.4$}
        \label{fig:u3_t24}
    \end{subfigure}
    
    \vspace{0.5cm}
    
    \begin{subfigure}[b]{0.32\textwidth}
        \centering
        \includegraphics[width=\textwidth]{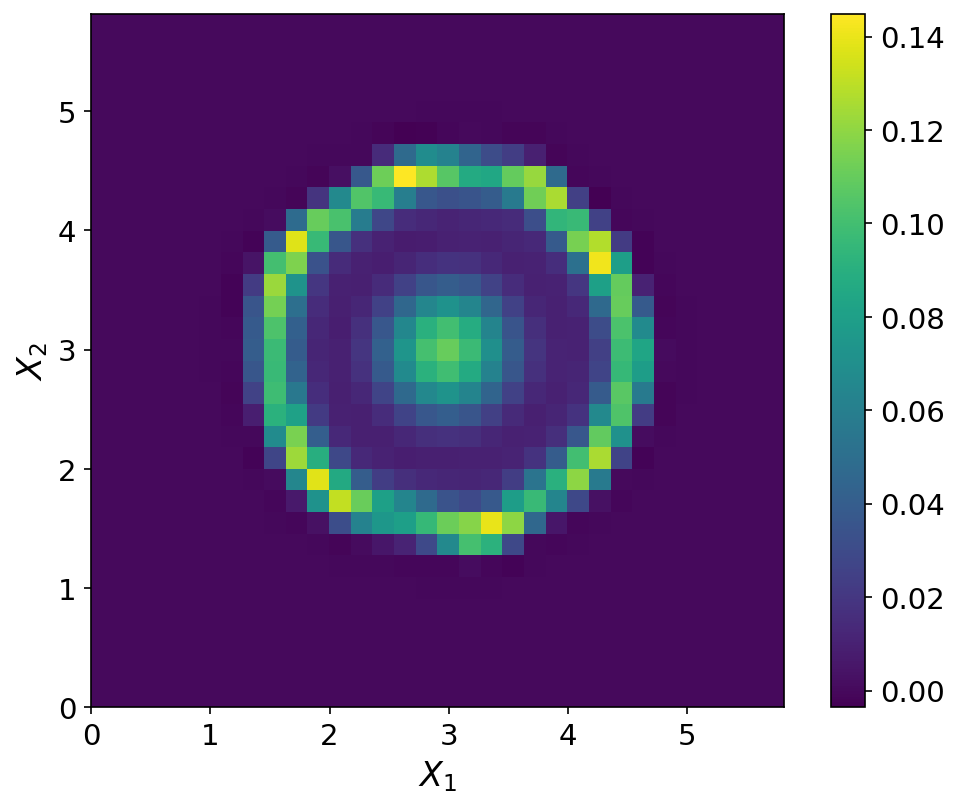}
        \caption{$t = 3.6$}
        \label{fig:u3_t36}
    \end{subfigure}
    \hfill
    \begin{subfigure}[b]{0.32\textwidth}
        \centering
        \includegraphics[width=\textwidth]{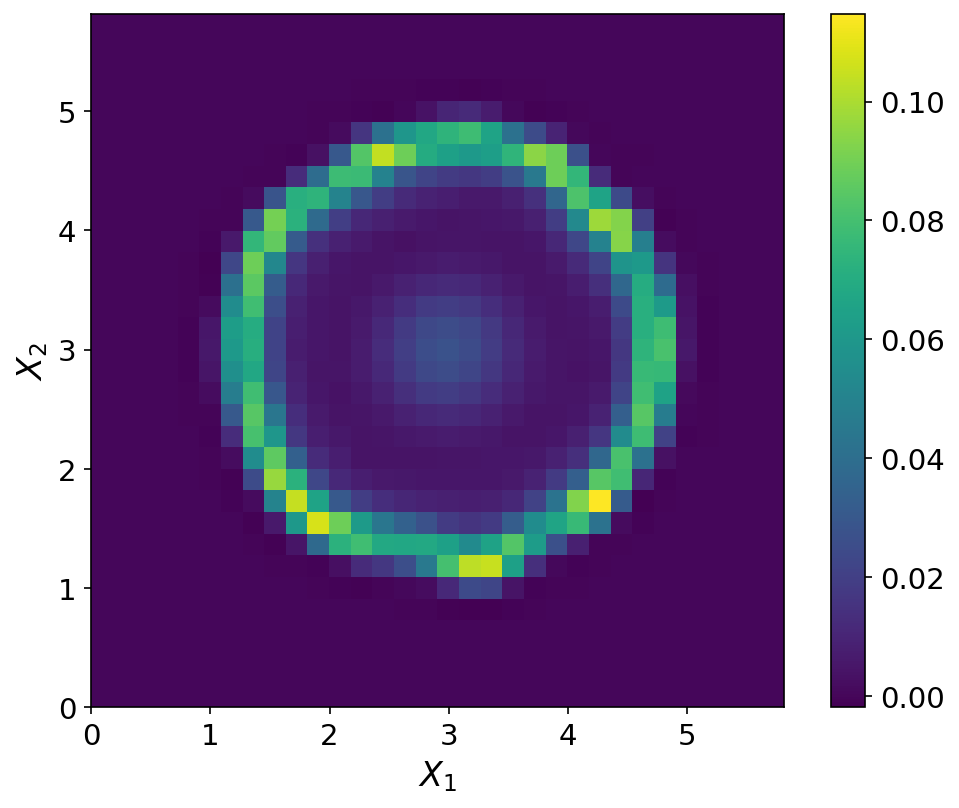}
        \caption{$t = 4.8$}
        \label{fig:u3_t48}
    \end{subfigure}
    \hfill
    \begin{subfigure}[b]{0.32\textwidth}
        \centering
        \includegraphics[width=\textwidth]{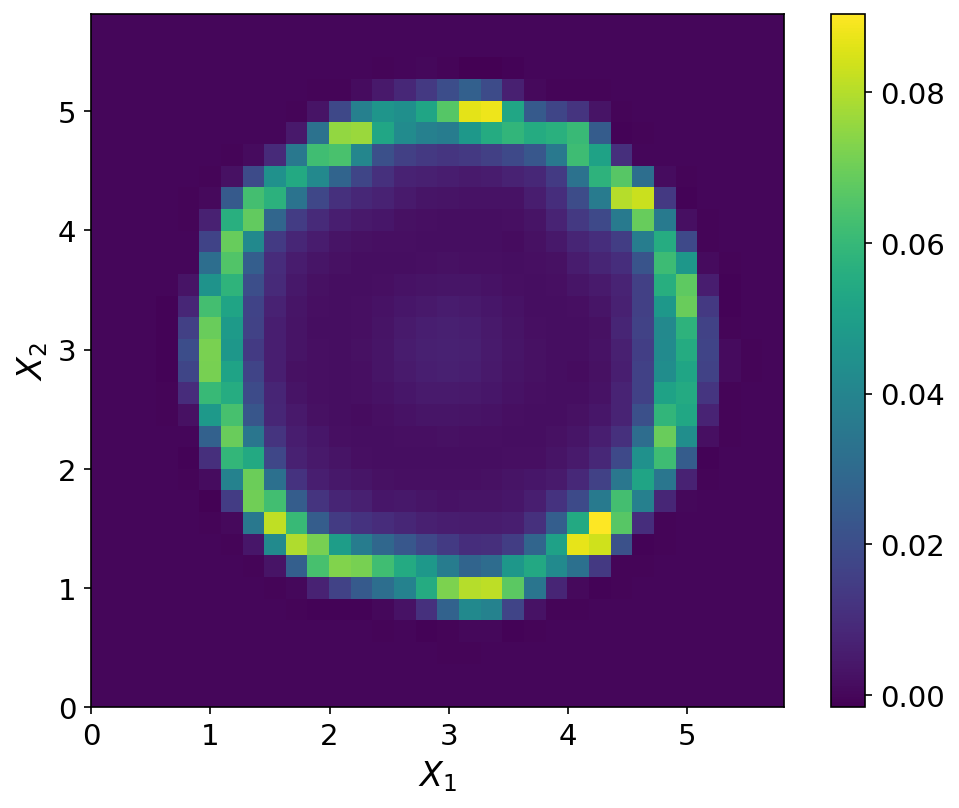}
        \caption{$t = 6.0$}
        \label{fig:u3_t60}
    \end{subfigure}
    
    \caption{Evolution of the cell density $u$ and the concentrations $v, m, w$ on the cross-section at $x_3 = L/2$.}
    \label{fig:3d_1}
\end{figure}

\begin{figure}[!ht]
    \centering
    
    \begin{subfigure}[b]{0.32\textwidth}
        \centering
        \includegraphics[width=\textwidth]{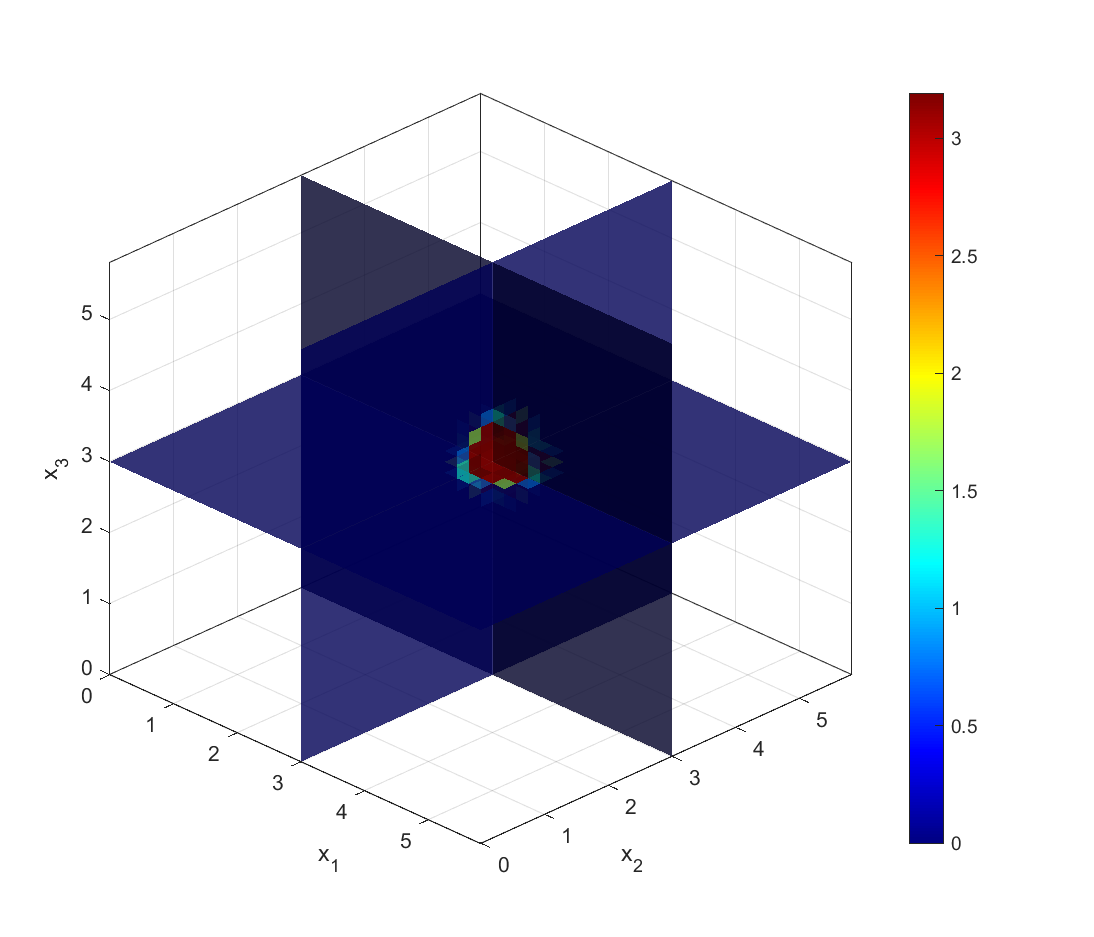}
        \caption{$t = 0.0$}
        \label{fig:u32_t00}
    \end{subfigure}
    \hfill
    \begin{subfigure}[b]{0.32\textwidth}
        \centering
        \includegraphics[width=\textwidth]{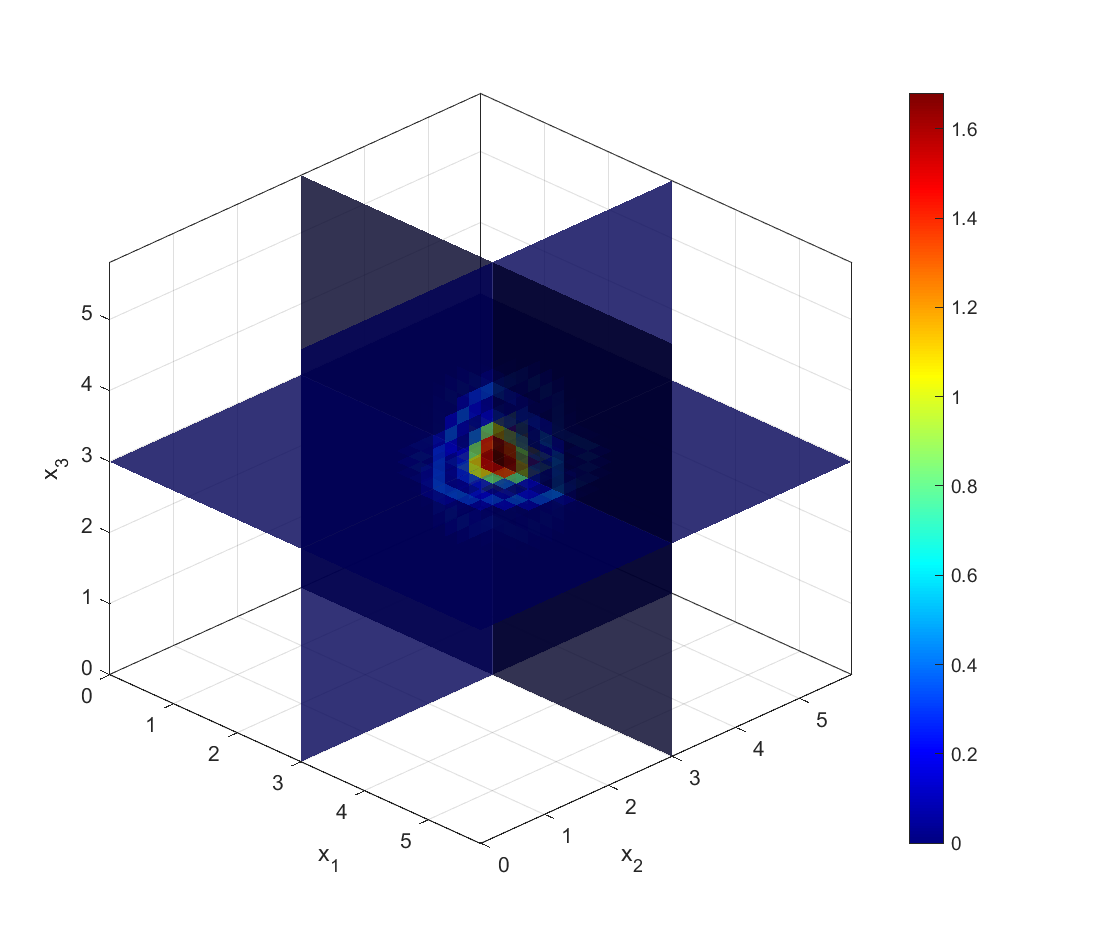}
        \caption{$t = 1.2$}
        \label{fig:u32_t12}
    \end{subfigure}
    \hfill
    \begin{subfigure}[b]{0.32\textwidth}
        \centering
        \includegraphics[width=\textwidth]{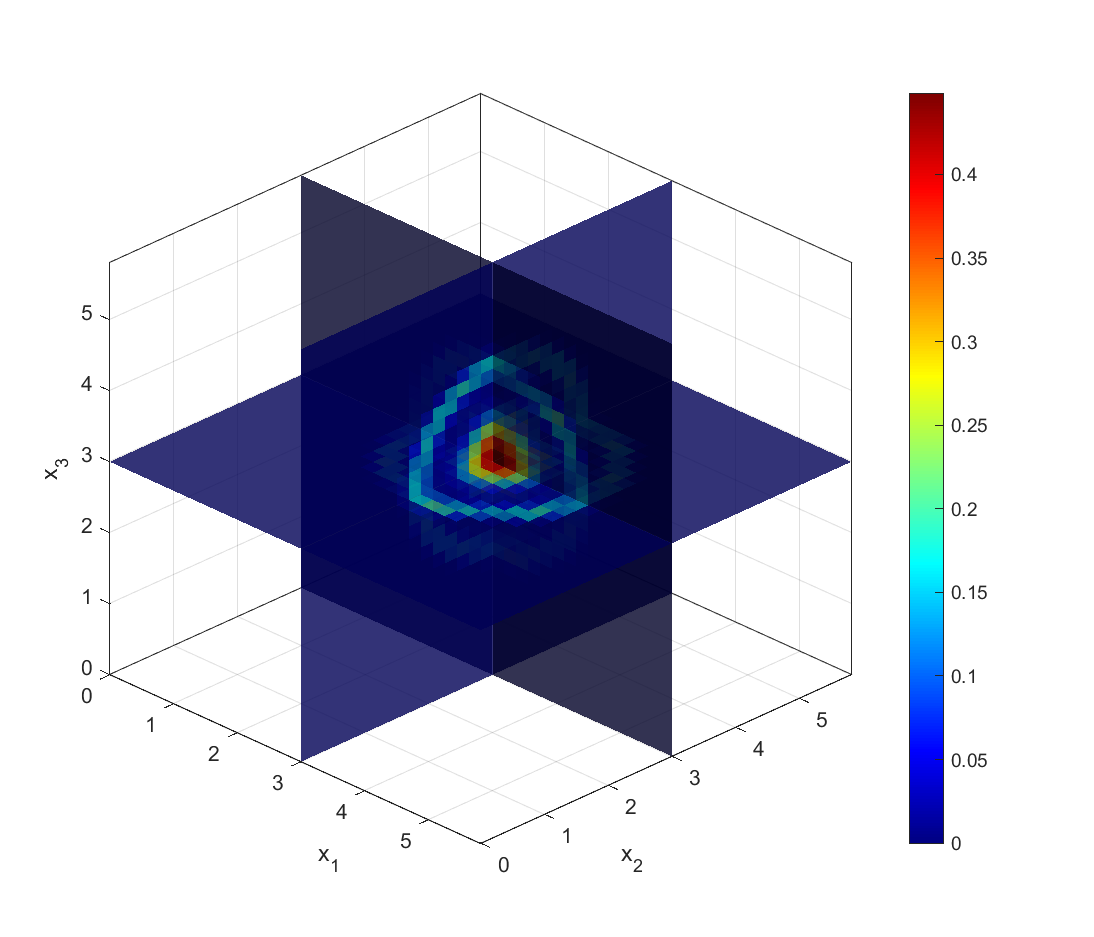}
        \caption{$t = 2.4$}
        \label{fig:u32_t24}
    \end{subfigure}
    
    \vspace{0.5cm}
    
    \begin{subfigure}[b]{0.32\textwidth}
        \centering
        \includegraphics[width=\textwidth]{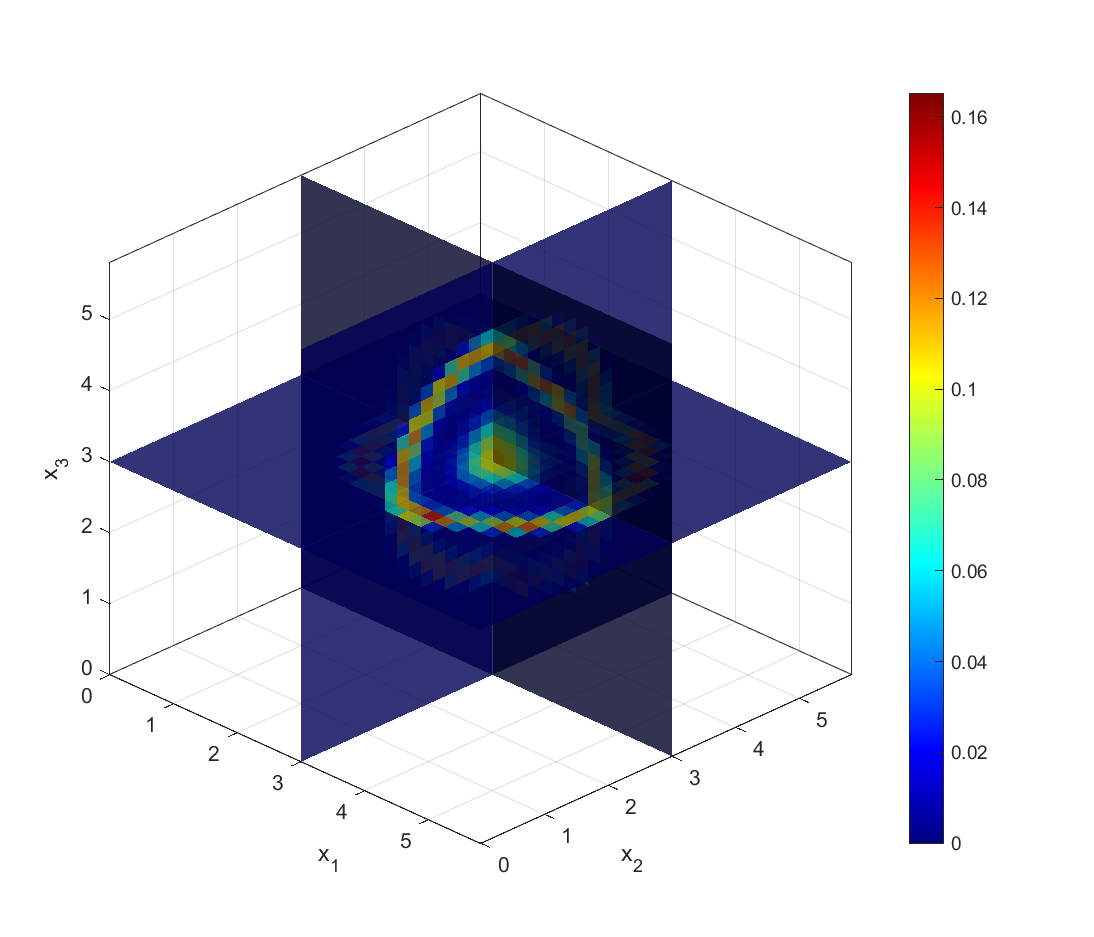}
        \caption{$t = 3.6$}
        \label{fig:u32_t36}
    \end{subfigure}
    \hfill
    \begin{subfigure}[b]{0.32\textwidth}
        \centering
        \includegraphics[width=\textwidth]{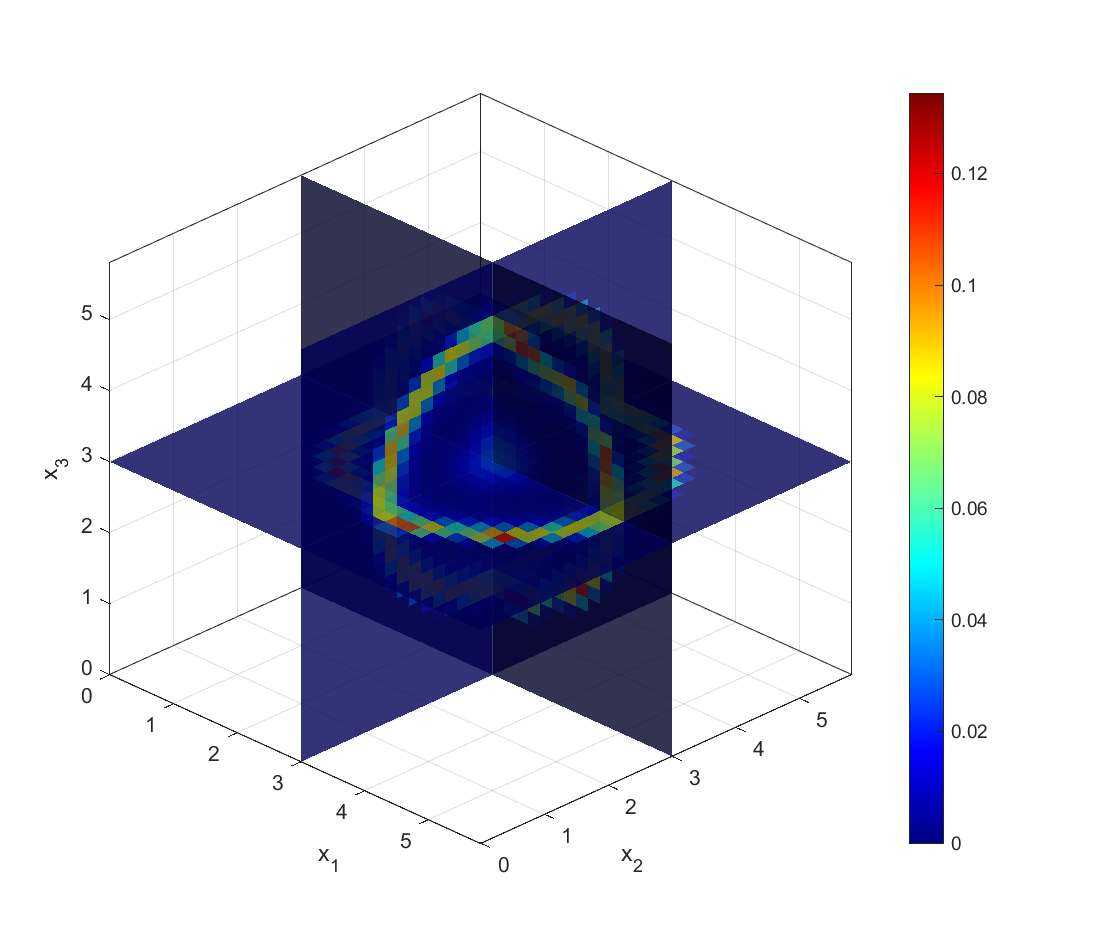}
        \caption{$t = 4.8$}
        \label{fig:u32_t48}
    \end{subfigure}
    \hfill
    \begin{subfigure}[b]{0.32\textwidth}
        \centering
        \includegraphics[width=\textwidth]{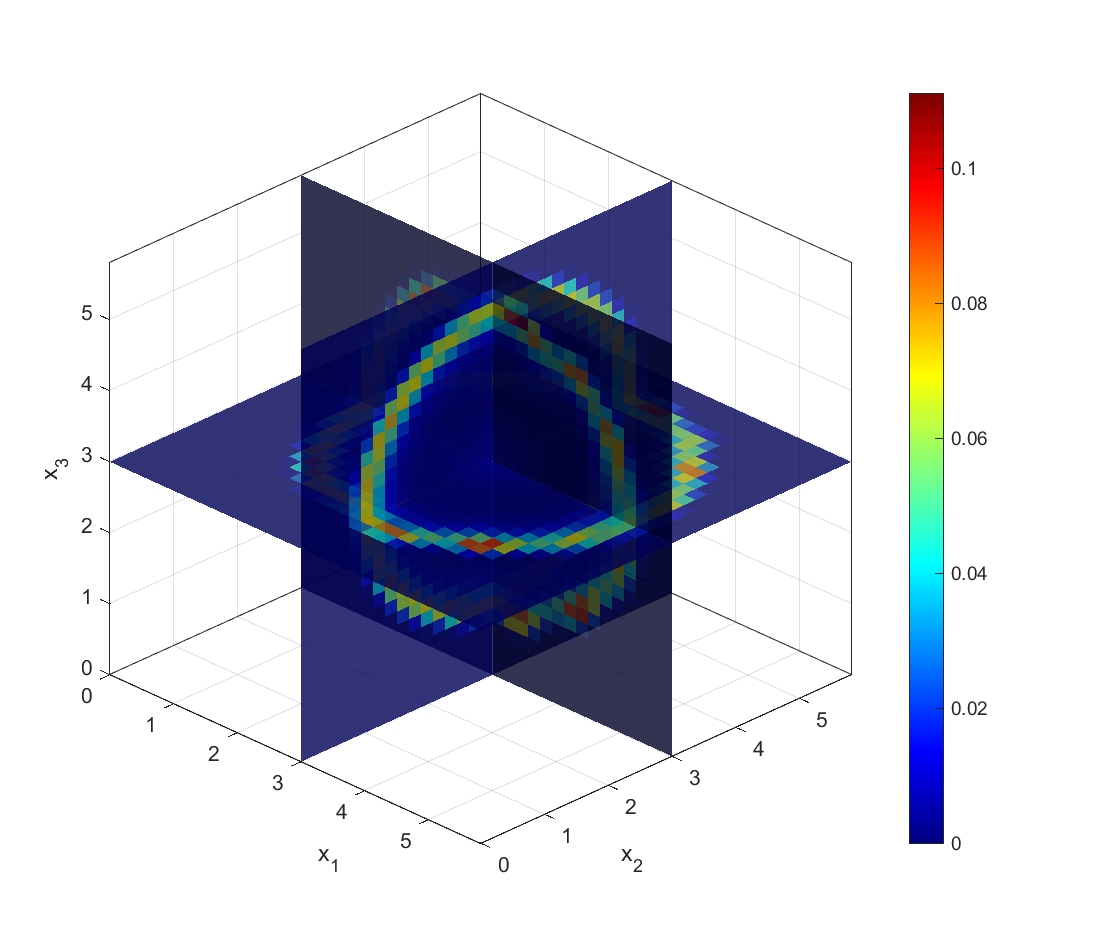}
        \caption{$t = 6.0$}
        \label{fig:u32_t60}
    \end{subfigure}
    
    \caption{Three-dimensional distribution of the cell density $u$ on mutually perpendicular cross-sections at $x_1 = L/2$, $x_2 = L/2$, and $x_3 = L/2$.}
    \label{fig:3d_2}
\end{figure}

It is noteworthy that this work presents the first systematic numerical computation of the three-dimensional space for \eqref{eq:haptotaxis-system}, and consequently, no alternative algorithms are available in the literature for comparative benchmarking. Nevertheless, the convergence of the solution as $H \to \infty$ in our algorithm can be validated by comparing the results obtained at the highest resolution \(H=2^8\) with other resolutions \(H=2^4,...,2^7\). As shown in Figure \ref{fig:3d_3}, the spatial convergence order is validated through the slope of $\log_2 \mathcal{E}$ versus $\log_2 H$, with a measured value of $-1.407$, which closely matches the theoretical order of $16/13$ (Theorem \ref{thm:overall}).
\begin{figure}
    \centering
    \includegraphics[width=0.45\linewidth]{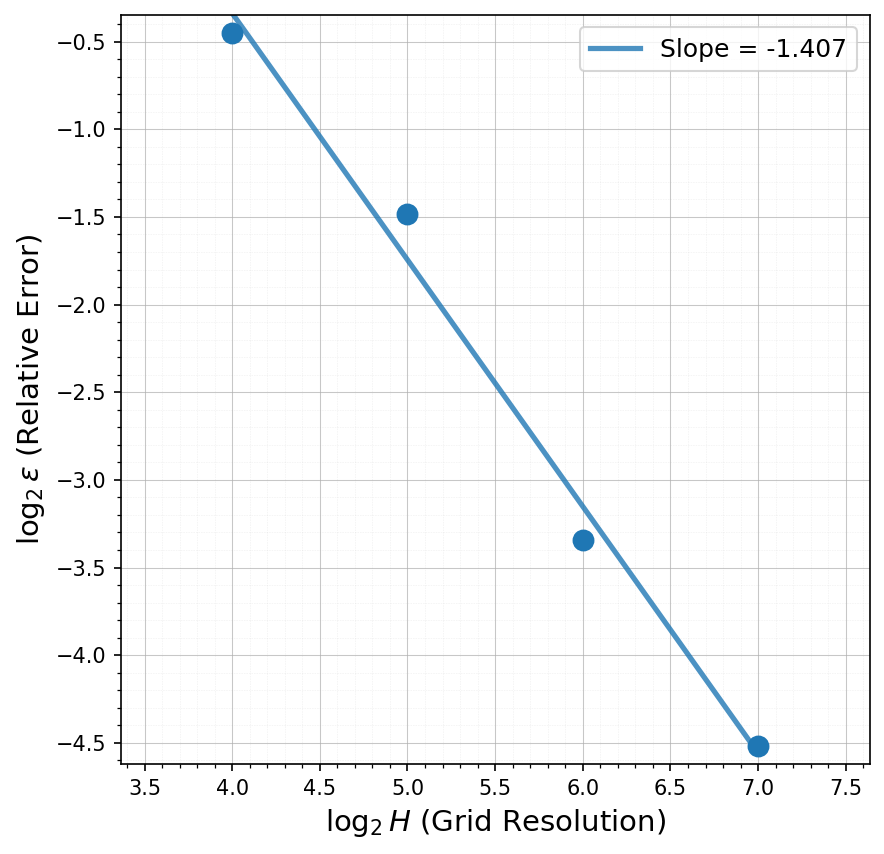}
    \caption{Relative error convergence against grid resolution ($H$).}
    \label{fig:3d_3}
\end{figure}

\section{Conclusions}
\label{sec:conclusion}
In this work, we have developed a novel numerical framework for solving a biological reaction-diffusion-advection system in three spatial dimensions using particles of variable mass. The cell density was represented through empirical particle measures, while the concentrations of several associated chemical species were constructed and dynamically updated on the spatial grid. Efficient particle-grid interaction was realized via Particle-in-Cell (PIC) interpolations, and spatial diffusion was handled with a spectral solver, ensuring both computational efficiency and high precision.

This new method achieves boundedness of the particle mass rate of change over finite time intervals, inherent non-negativity of cell density through the empirical measure representation, and unconditional positivity preservation of all concentration fields on the spatial grid. A rigorous error analysis confirms the theoretical convergence rates, which are further validated by numerical experiments.

Our method successfully captured the rapid spread of cells driven by haptotactic flux in a three-dimensional setting, consistent with the patterns reported in two-dimensional cases. Future work will explore extensions to multi-species cell populations, and additional biophysical mechanisms such as cell-cell adhesion and mechano-chemical feedback. The spectral treatment of diffusion and the positivity-preserving properties of the scheme also suggest potential applications beyond the biological context considered here, offering a versatile platform for a broader class of reaction-diffusion-advection problems in three dimensions.

\section{Acknowledgements}
\noindent ZZ were supported by the National Natural Science Foundation of China (Projects 92470103), the Hong Kong RGC grant (Projects 17304324 and 17300325), the Seed Funding Programme for Basic Research (HKU), and the Hong Kong RGC Research Fellow Scheme 2025. ZW was partially supported by NTU SUG-023162-00001, MOE AcRF Tier 1 Grant RG17/24. JX was partially supported by NSF grants DMS-2219904 and DMS-2309520. The computations were performed at the research computing facilities provided by the Information Technology Services, the University of Hong Kong, and the Greenplanet Cluster at UC Irvine.

\end{document}